\documentclass[a4paper,10pt]{article}

\usepackage{amssymb, amsmath, mathtools, amsfonts, times, amsthm, amscd, setspace, leftidx, enumitem, xfrac, xcolor}
\usepackage{geometry}
\input xy 
\xyoption{all}
\usepackage{scalerel}

\usepackage[titletoc,title]{appendix}
\usepackage{graphicx}
\graphicspath{ {Graphics/} }
\usepackage{pdflscape}

\usepackage{array,multirow}
\theoremstyle{plain}

\newtheorem{proposition}{Proposition}[section]
\newtheorem{defi}[proposition]{Definition}
\newtheorem{lemma}[proposition]{Lemma}
\newtheorem{thm}[proposition]{Theorem}
\newtheorem{corollary}[proposition]{Corollary}
\newtheorem{rmk}[proposition]{Remark}
\newtheorem{prop}[proposition]{Proposition}
\newtheorem{ex}[proposition]{Example}

\newcommand{\cnc}[0]{\nabla}
\newcommand{\ct}[1]{\mathcal{#1}}
\newcommand{\ov}[1]{\overline{#1}}

\newcommand{\K}[0]{\mathbb{K}}

\newcommand{\XA}[0]{\mathfrak{X}}

\newcommand{\sbt}[0]{{\scriptstyle\bullet}}

\newcommand{\bim}[1]{\prescript{}{#1}{\mathcal{M}}_{#1}}
\newcommand{\fr}[1]{\mathfrak{#1}}
\newcommand{\tn}[0]{\otimes}

\newcommand{\op}[0]{\mathrm{op}}
\newcommand{\cvl}[0]{\mathrm{coev}}
\newcommand{\evl}[0]{\mathrm{ev}}
\newcommand{\cvr}[0]{\underline{\mathrm{coev}}}
\newcommand{\evr}[0]{\underline{\mathrm{ev}}}

\newcommand{\id}[0]{\mathrm{id}}
\newcommand{\Hom}[0]{\mathrm{Hom}}
\newcommand{\flip}[0]{\mathfrak{flip}}

\newcommand{\lmod}[1]{\prescript{}{#1}{\mathcal{M}}}

\newcommand{\source}[0]{\mathsf{s}}
\newcommand{\target}[0]{\mathsf{t}}

\newcommand{\ab}[1]{{#1}_{\mathrm{ab}}}
\newcommand{\iso}[1]{{#1}_{\mathrm{iso}}}

\newcommand{\arr}[2]{{#1}\rightarrow{#2}}
\newcommand{\arl}[2]{{#2}\leftarrow{#1}}
\newcommand{\arrtwo}[3]{{#1}\scaleto{\rightarrow}{1.7pt} {#2}\scaleto{\rightarrow}{1.7pt} {#3} }
\newcommand{\arltwo}[3]{{#3}\leftarrow{#2}\leftarrow{#1}}
\newcommand{\dist}[2]{{\scriptscriptstyle({#1},{#2})}}

\newcommand{\ls}[1]{{}_{{\source}} {#1}}
\newcommand{\lt}[1]{{}_{{\target}} {#1}}
\newcommand{\rs}[1]{{#1}{}_{{\source}} }
\newcommand{\rt}[1]{{#1}{}_{{\target}}}
\newcommand{\tM}[1]{\lt{#1}}

\newcommand{\Mto}[1]{{#1}{}_{{{\target}^{\op}}}}

\newcommand{\sMt}[1]{{}_{{\source}} {#1}{}_{{\target}}}

\newcommand{\sMs}[1]{{}_{{\source}} {#1}{}_{{\source}}}
\newcommand{\tMt}[1]{{}_{{\target}} {#1}{}_{{\target}}}
\newcommand{\sMto}[1]{{}_{{\source}} {#1}{}_{t^{\op}}}

\newcommand{\genleft}[4]{\ooalign{\raisebox{-0.6ex}{$\scriptstyle \arl{#3}{#4}$}\cr\raisebox{0.6ex}{$\scriptstyle \arl{#1}{#2}$}}}
\newcommand{\genright}[4]{\ooalign{\raisebox{-0.6ex}{$\scriptstyle \arr{#3}{#4}$}\cr\raisebox{0.6ex}{$\scriptstyle \arr{#1}{#2}$}}}
\newcommand{\genlefty}[3]{\ooalign{\raisebox{-0.6ex}{$\scriptstyle \hspace{2mm} #3$}\cr\raisebox{0.6ex}{$\scriptstyle \arl{#1}{#2}$}}}
\newcommand{\genrighty}[3]{\ooalign{\raisebox{-0.6ex}{$\scriptstyle \arr{#1}{#2}$}\cr\raisebox{0.6ex}{$\scriptstyle \hspace{2mm} #3$}}}
\newcommand{\points}[2]{\left(\ooalign{\raisebox{-0.6ex}{$\scriptstyle #2$}\cr\raisebox{0.6ex}{$\scriptstyle #1$}}\right)}
\allowdisplaybreaks
\begin{document}

\title{Isotopy Quotients of Hopf Algebroids and the Fundamental Groupoid of Digraphs}
\author{Aryan Ghobadi \\ \small{Queen Mary University of London }\\\small{ School of Mathematics, Mile End Road}\\\small{ London E1 4NS, UK }\\ \small{Email: a.ghobadi@qmul.ac.uk}}
\date{}

\maketitle
\begin{abstract}
We build on our construction of Hopf algebroids from noncommutative calculi under the further assumption of surjectivity for the calculus. We also introduce the notions of Hopf ideals and isotopy quotients for arbitrary Hopf algebroids. Using these ingredients, we prove a Riemann-Hilbert correspondence for digraphs, by showing that the groupoid algebra of the fundamental groupoid of a digraph is isomorphic to the isotopy quotient of the Hopf algberoid corresponding to flat connections over the digraph.
\end{abstract}
\begin{footnotesize}2020\textit{ Mathematics Subject Classification}: Primary 16T99, 58B32, 05C20 ; Secondary 20L05, 16D25
\end{footnotesize}
\section{Introduction} 
The standard Riemann-Hilbert correspondence for smooth manifolds refers to the equivalence between the category of vector bundles with flat connections on a manifold and the category of representations of the fundamental path groupoid of the manifold. In the setting of noncommutative differential geometry \cite{beggs2020quantum}, one replaces a manifold with an arbitrary DGA and can still define an analogous category of flat bimodule connections over this `noncommutative space'. Moreover, we showed in \cite{ghobadi2020hopf} how to construct a Hopf algebroid $\ct{D}\XA$, which represents this category of connections and hence takes a somewhat analogous role to the classical algebra of differential operators. However, unlike the classical picture, it is far from clear in general how to extract homotopical data from this categorical approach. In this work, we show the first instance of $\ct{D}\XA$ carrying homotopical information, by focusing on the differential calculi associated to digraphs \cite{beggs2020quantum,majid2013noncommutative}. Concretely, we show that the \emph{isotopy quotient} of $\ct{D}\XA$ obtained from a digraph's differential calculus, recovers the fundamental groupoid of the digraph, as defined by Grigor'yan, Jimenez and Muranov \cite{grigor2018fundamental}. The key new aspect of our theory is the introduction of isotopy quotients of Hopf algebroids, which emerges from the classical limit where the input DGA is the de Rham complex of a smooth manifold. In this scenario, $\ct{D}\XA$ itself is too large but its isotopy quotient is the usual algebra of differential operators on the manifold representing the classical category of flat connections, and is thereby Morita equivalent to the fundamental path groupoid. Hence, we present the isotopy quotient of $\ct{D}\XA$ as our first tool for obtaining homotopical data from flat connections in the noncommutative world. 

The inspiration for defining the isotropy quotient of a Hopf algebroids comes from the theory of groupoids. A groupoid $\ct{G}$ has two components, a set of base points, $\ct{G}_{0}$, and a set of arrows between these points, $\ct{G}_{1}$. Similarly, the data for a Hopf algebroid consists of a pair of algebras $(A,H)$, where $A$ is called the base algebra and the structural morphisms on $H$, while similar to that of a Hopf algebra, are defined with respect to $A$. In fact, commutative Hopf algebroids correspond to representable presheafs of groupoids on the category of affine schemes, see Example \ref{ExPresheaf}. Hence, one can generalise various constructions related to groupoids to the setting of Hopf algebroids. For example, for any groupoid, we can consider its \emph{isotopy sub-groupoid},  which has the same base space, but only contains arrows with the same source and target. In Section \ref{SIso}, we define the corresponding notion of \emph{isotopy quotient} for Hopf algebroids, whose relation with the isotopy sub-groupoid is clarified in Examples \ref{Ex:IsoGrpd} and \ref{Ex:IsoFinGrpd}. Dual to the groupoid picture, we obtain the isotopy quotient of a Hopf algebroid $(A,H)$ by setting the image of the source and targets to be equal. The resulting algebra will no longer be a Hopf algebroid over the original base algebra, $A$, but over its abelianisation, $\ab{A}$. To obtain this result, we dedicate the first half of Section \ref{SIdeals} to developing the appropriate theory of ideals and quotients for Hopf algebroids. Despite the extensive literature on these objects, the author is not aware of any other works studying ideals for arbitrary Hopf algebroids. The reader should note that unlike Hopf algebras, Hopf algebroids do not admit antipodes in general, and the Hopf condition for an ideal in this case, \eqref{eq:HpfIdeal}, can not be simplified to stability of the ideal under the antipode. 

The notion of isotopy quotient becomes of special significance in our study of $\ct{D}\XA$, since it relates this Hopf algebroid and the universal enveloping algebroid of a Lie-Rinehart algebra, see Section \ref{SBialgIso}. While this connection was noted in \cite{ghobadi2020hopf}, it was only at the level of algebras, and we did not formally prove that we can recover the Hopf algebroid structure of the enveloping algebroid in this way. Our work in Section \ref{SIso} now fills this gap. Moreover, we dedicate Section \ref{SSurjective} to expanding our construction of $\ct{D}\XA$ under the assumption that the calculus is surjective. This is a natural condition when looking back at differential geometry and is usually assumed. The reason behind not assuming surjectivity in \cite{ghobadi2020hopf} was providing a construction of Hopf algebroids, with minimal assumptions. In particular, $H\XA^{1}$ and $\ct{D}\XA$ are rare examples of noncommutative and noncocommutative Hopf algebroids over arbitrary base algebras, which do not arise from Tannaka-Krein type reconstruction as in \cite{hai2008tannaka,shimizu2019tannaka,szlachanyi2009fiber}. However, the downside of working with these algebras is that they are rather large in general. In Section \ref{SSurjective}, we show that their description can be thoroughly simplified when assuming that the calculus is surjective.

Using the work done in Section \ref{SSurjective}, we go on to describe the Hopf algebroids $H\XA^{1}$ and $\ct{D}\XA$ for the calculus of a given digraph. For this description we introduce the following notion: given a digraph $D=(V,E)$, we define its \emph{double}, denoted by $D^{\mathrm{db}}$ as the digraph with $V\times V$ as its vertex set and the following families of edges:
\begin{itemize}
\itemsep0em 
\item An edge from $(a,c)$ to $(b,d)$ corresponding to any pair of edges $\arr{a}{b}, \arr{c}{d}\in E$
\item An edge from $(b,d)$ to $(a,c)$ corresponding to any pair of edges $\arr{a}{b}, \arr{c}{d}\in E$
\item An edge from $(a,c)$ to $(b,c)$ corresponding to any edge $\arr{a}{b}\in E$ and any vertex $c\in V$
\item An edge from $(a,d)$ to $(a,c)$ corresponding to any edge $\arr{c}{d}\in E$ and any vertex $a\in V$
\end{itemize} 
In the usual terminology of digraphs we can decompose the set of edges of $D^{\mathrm{db}}$ as the union of the edges of  $D\square D^{-1}$, $D\times D$ and $D^{-1}\square D^{-1}$, where $D^{-1}$ has the same vertex set of $D$ with reversed edges, see Chapter 10 of \cite{bang2018classes}. In Theorems \ref{THDig} and \ref{TDXDigraph}, we demonstrate that the algebras $H\XA^{1}$ and $\ct{D}\XA$ obtained from the digraph calculus, can be viewed as quotients of the path algebra of the double of $D$, denoted by $\K D^{\mathrm{db}}$, respectively. In \cite{ghobadi2020hopf}, we already demonstrated that for a digraph calculus, Beggs and Brzezi{\'n}ski's algebra of vector fields, denoted by $T\XA^{1}_{\sbt}$ in \cite{beggs2014noncommutative}, would be isomorphic to the path algebra of the digraph, and also provided a closely related description $H\XA^{1}$ and $\ct{D}\XA$ for an arbitrary quiver. However, our description of $H\XA^{1}$ here goes beyond that of \cite{ghobadi2020hopf} and our result for $\ct{D}\XA$ completely differs from that of \cite{ghobadi2020hopf} due to our novel choice of $\Omega^{2}_{\mathrm{top}}$ for extending the first-order differential calculus of the digraph. See Remark \ref{ROmega2s} for additional details on the latter point. 

Finally, we prove our Riemann-Hilbert type result in Section \ref{SIsoDigraph}. The fundamental groupoid, $\Pi_{D}$, of a digraph, $D$, was introduced in \cite{grigor2018fundamental}, where its homotopical properties are studied. The groupoid $\Pi_{D}$ has $V$ as its set of base points, and is defined as the quotient of the free groupoid generated by the arrows $\arr{a}{b}\in E$, which are denoted by $(\arl{a}{b})$ in the groupoid, where inverses, $(\arr{a}{b})$, are formally added, by the following two types of relations:
\begin{itemize}
\itemsep0em 
\item For any triangle $(p,q,r)$ in the digraph, where $\arr{p}{q},\arr{q}{r},\arr{p}{r}\in E$, we have $(\arl{q}{r}).(\arl{p}{q})= (\arl{p}{r})$
\item Given any square $(p,q,q',r)$ in the digraph, where $\arr{p}{q},\arr{q}{r},\arr{p}{q'},\arr{q'}{r}\in E$, we have $(\arl{q}{r}).(\arl{p}{q})= (\arl{q'}{r}).(\arl{p}{q'})$
\end{itemize}
Given such a groupoid we obtain a natural Hopf algebroid called its groupoid algebra, Example \ref{ExGrpd}. In Corollary \ref{CDXFund}, we show that the isotopy quotient of $\ct{D}\XA$ is isomorphic to the groupoid algebra $\K \Pi_{D}$.

\textbf{Organisation.} In Section \ref{SIdeals} we introduce the theory of Hopf ideals and isotopy quotients for Hopf algebroids. This section can be read on its own and only requires background on Hopf algebroids which is present in preliminary Section \ref{SPrelimHopf}. Preliminary Section \ref{SPrelimNCG} provides a review of the main results of \cite{ghobadi2020hopf} and how they will be adapted under the assumption of the calculus being surjective in Section \ref{SSurjective}. In Section \ref{SGrph} we focus on the specific applications of our work for digraph calculi.

\textbf{Acknowledgements.} The author would firstly like to thank Laiachi El Kaoutit and the University of Granada for their hospitality during the author's visit in September 2020, where much of Section \ref{SIdeals} was developed, as well as the LMS for funding this research visit under the early career research grants, reference number ECR-1920-42. The author would also like thank Shahn Majid for numerous fruitful conversations regarding this work and Paolo Saracco for helpful comments. 
\section{Preliminaries} 
\textbf{Notation.} Throughout this work $\K$ will denote an arbitrary field and $A$ a $\K$-algebra. We adapt the notion of \cite{ghobadi2020hopf} and denote the multiplication in $H\XA^{1}$ and $\ct{D}\XA$ by $\sbt$, while multiplication in all other algebras will be denoted by the standard $.$, or $m$ when viewed as a map. We only consider small groupoids in this work and will use the notation $\ct{G}= (\ct{G}_{0},\ct{G}_{1})$, where $\ct{G}_{0}$ denotes the set of objects and $\ct{G}_{1}$ the set of morphisms, but we will refer to elements of $\ct{G}_{1}$ as arrows of the groupoid rather than morphisms. In the context of a groupoid, we denote the composition of arrows by $.$ and the standard source and target maps by $s,t:\ct{G}_{1}\rightarrow \ct{G}_{0}$. 
\subsection{Hopf algebroids constructed from NCG}\label{SPrelimNCG}
In \cite{ghobadi2020hopf}, we constructed several bialgebroids and Hopf algebroids representing various subcategories of bimodule connections corresponding to suitable differential calculi. In this section, we will assume the reader is familiar with basic notions such as bimodule connections and refer the reader to \cite{beggs2020quantum} and \cite{ghobadi2020hopf} for additional details. We will present a summary of our results in \cite{ghobadi2020hopf} and provide a picture of how the work in Section \ref{SSurjective} fits within this framework. 

First, let us review the ingredients for our construction. Let $(A,\oplus_{0 \leq i}\Omega^{n},d,\wedge)$ be a DGA. The additional datum that our construction in \cite{ghobadi2020hopf} requires are the following: 
\begin{enumerate}[label=(\Roman*), leftmargin=0.7cm] 
\item We assume that $\Omega^{1}$ and $\Omega^{2}$ are Frobenius as $A$-bimodules. This is equivalent to saying that for each $n\in \lbrace 1,2 \rbrace$ there exist a bimodule $\XA^{n}\cong \mathrm{Hom}_{A}(\Omega^{n},A)\cong \prescript{}{A}{\mathrm{Hom}}(\Omega^{n},A) $, which is both left and right dual to $\Omega^{n}$ in the monoidal category $\bim{A}$. Additionally, we choose a family of duality morphisms 
\begin{align}
\cvl^{n}: A \rightarrow \Omega^{n}\otimes\XA^{n}, \quad & \quad \evl: \XA^{n}\otimes\Omega^{n} \rightarrow A
\\ \cvr^{n}: A \rightarrow \XA^{n}\otimes\Omega^{n}, \quad &\quad \evr^{n}: \Omega^{n}\otimes \XA^{n}\rightarrow A
\end{align}
for $n\in \lbrace 1,2 \rbrace$ and adapt the notation 
\begin{equation} \cvl^{1}(1)= \sum_{i} \omega_{i}\otimes x_{i}, \quad \quad \cvr^{1}(1)= \sum_{j} y_{j}\otimes \rho_{j}
\end{equation}
We will avoid writing $\sum_{i}$ and $\sum_{j}$ and whenever pairs $x_{i}$, $\omega_{i}$ and $\rho_{j}$, $y_{j}$ appear with matching indices, summation will be implicitly assumed.
\item We assume that $\wedge :\Omega^{1}\tn \Omega^{1}\rightarrow \Omega^{2}$ is a pivotal morphism. Explicitly, this means that for any $x^{2}\in \XA^{2}$, the equation
\begin{equation}\label{EqPivWedge}
\evl^{2} (x^{2} \tn \omega_{i}\wedge \omega_{j}) x_{j}\otimes x_{i} = y_{i}\otimes y_{j} \evr^{2}(\rho_{j}\wedge \rho_{i} \otimes x^{2} ) 
\end{equation}
holds.
\end{enumerate}
Describing (invertible/extendible) bimodule connections $(M,\cnc , \sigma)$, as done in \cite{beggs2020quantum}, does not require the additional data of the duality morphisms on the DGA. However in \cite{ghobadi2020hopf}, we consider bimodule connections where $\sigma$ is compatible with this additional data. To distinguish between these settings we will call the additional data of the duality morphisms on the DGA a \emph{pivotal structure}, $\Pi$, on it. Hence, we can make the following definition. 

\begin{defi}[\cite{ghobadi2020hopf}]\label{DABC} If $\Pi$ is a pivotal structure on $(A,\oplus_{0 \leq i}\Omega^{n},d,\wedge)$, we call an invertible bimodule connection  $(M,\cnc , \sigma)$ a $\Pi$-\emph{adaptable bimodule connection}, or $\Pi$-\emph{ABC} for short, if the induced morphisms 
\begin{align}
(\evl^{n}\otimes \id_{M\otimes \XA^{n}})(\id_{\XA^{n}}\otimes\sigma\otimes \id_{\XA^{n}})(\id_{\XA^{n}\otimes M}\otimes \cvl^{n}): \XA^{n}\otimes M\rightarrow M\otimes\XA^{n} &\label{EqsigX}
\\ (\id_{\XA^{n}\otimes M}\otimes \evr^{n} )(\mathrm{id}_{\XA^{n}}\otimes\sigma^{-1}\otimes \id_{\XA^{n}})(\cvr^{n}\otimes \id_{M\otimes \XA^{n}}): M\otimes\XA^{n} \rightarrow \XA^{n}\otimes M &\label{EqXsig}
\end{align}
are inverses for $n=1$. If $\Omega^{2}$ and $\wedge$ are as above, we call an extendible bimodule connection, $\Pi$-adaptable if the morphisms \eqref{EqsigX} and \eqref{EqXsig} are inverses for $n=2$. 
\end{defi}

Observe that given an invertible bimodule morphism $\sigma: M\otimes \Omega^{n}\rightarrow \Omega^{n}\otimes M$, the morphisms \eqref{EqsigX} and \eqref{EqXsig} being inverses is equivalent to stating that there exists an invertible morphism $\sigma_{\XA}: \XA^{n}\otimes M\rightarrow M\otimes\XA^{n}$ such that the duality morphisms \emph{braid} past $M$, via $\sigma$ and $\sigma_{\XA}$. More detail on this aspect of the theory can be found in Section 6.2 of \cite{ghobadi2021pivotal}.

Recall that for any algebra $R$, we can consider the free monoid generated by a $R$-bimodule $M$ in $\bim{R}$, on the space $R\oplus M\oplus M\tn_{R} M \oplus \cdots $, which we will denote by $T_{R}M$. Additionally, we denote the free product (coproduct in the category of $\K$-algebras) of two algebras $R$ and $S$ by $R\star S$ and consider the algebras $A^{e}$, $T_{\K} \XA^{1}$, $R_{1}^{n}:= T_{\K}(\Omega^{n}\tn_{\K} \XA^{n})$ and $R_{2}^{n}:= T_{\K}(\XA^{n}\tn_{\K} \Omega^{n})$ for $n\in \lbrace 1,2 \rbrace$. As an algebra, the Hopf algebroid $\ct{D}\XA$ is a quotient of the free product of the mentioned algebras by a series of relations. We denote the generators of the mentioned algebras by elements of the form $a\ov{b}\in A^{e}$, $x \in \XA^{1}$, $(x,\omega )\in \XA^{1}\tn_{\K} \Omega^{1}$ and $(\omega,x)\in \Omega^{n}\tn_{\K} \XA^{n}$, respectively, and denote the multiplication in $\ct{D}\XA$ by $\sbt$. With this notation, $\ct{D}\XA$ is defined by the following relations 
\begin{align}
 a\sbt x=ax& \label{EqAX}
 \\  x\sbt a =xa+&\mathrm{ev}(x,da)\label{EqXA}
 \\ x\sbt\ov{a}=\ov{a}\sbt x+& (x,da) \label{EqBXA}
\\ a\ov{a'}\sbt(x^{n},\omega^{n})\sbt b\ov{b'}= (ax^{n}b, b'\omega^{n} a') \quad \quad& a\ov{a'}\sbt(\omega^{n},x^{n})\sbt b\ov{b'}= (a\omega^{n} b, b'x^{n} a')\label{EqYFrak}
\\ (\omega_{i},y)\sbt (x_{i},\omega )=&\ov{\underline{\mathrm{ev}}(\omega\otimes y)} \label{EqRelInv1}
\\   (x,\rho_{j})\sbt(\omega, y_{j})=&\mathrm{ev}(x\otimes\omega)\label{EqRelInv2} 
 \\ (y_{j},\omega)\sbt (\rho_{j},x)=&\ov{\mathrm{ev}(x\tn\omega)} \label{EqRelHpf1}
\\  (\omega, x_{i})\sbt(x,\omega_{i})=&\underline{\mathrm{ev}}(\omega\tn x)\label{EqRelHpf2}
\\ \fr{ev}(x^{2}\otimes d\omega_{i})\sbt x_{i} - \mathfrak{ev}(x^{2}\otimes &\omega_{j}\wedge\omega_{k}) \sbt x_{k}\sbt x_{j} =0\label{EqFlat}
\\\mathfrak{ev}( x^{2}\otimes \omega_{i}\wedge \omega_{j}) [x_{j}\sbt(x_{i},\omega)+(x_{j},\omega)\sbt x_{i}]& = \mathfrak{ev}(x^{2}\otimes d\omega_{i} )(x_{i},\omega ) - (x^{2},d\omega) \label{EqFlExt}
\\ \mathfrak{ev}(x^{2}\otimes \omega_{i}\wedge \omega_{j}) \sbt (x_{j},\rho)\sbt(x_{i},\omega)&= (x^{2},\omega\wedge \rho) \label{EqExt1}
\\ \ov{\underline{\mathfrak{ev}}(\rho_{i}\wedge \rho_{j}\otimes x^{2})}\sbt(\omega ,y_{i})\sbt(\rho, y_{j} ) &=(\omega\wedge\rho , x^{2})\label{EqExt2}
\end{align}
where $x^{n}\in \XA^{n}$ and $\omega^{n}\in \Omega^{n}$ for $n=1,2$. In \cite{ghobadi2020hopf}, we build up to the construction of $\ct{D}\XA$, by considering several other subcategories of bimodule connections, all of which can be represented by an algebra with some of the above relations. In the following table we summarise the notations of \cite{ghobadi2020hopf} and compare the presentation of the various algebras computed for an arbitrary DGA, to the simplified presentations of Section \ref{SSurjective}, where the calculus is assumed to be surjective.

\begin{center}
\begin{table}[h]
\begin{tabular}{ |c|c|c|c|c|} 
 \hline
\multirow{2}{*}{H/B}& \multirow{2}{*}{Algebra} & \multirow{2}{*}{Generators} & Relations  & Category of   \\ 
 & & &{\small (As numbered in this work)} & Representations  \\
\hline\hline 
 \multirow{2}{*}{-} & \multirow{2}{*}{$T\XA^{1}_{\sbt}$} & \multirow{2}{*}{$A\star T_{\K}\XA^{1}$} & \multirow{2}{*}{\eqref{EqAX}, \eqref{EqXA}} & left connections \\
 & & & & [Proposition 6.15 \cite{beggs2020quantum}]\\
 \hline 
\multirow{3}{*}{B} & \multirow{3}{*}{$B\XA^{1}$} & $A^{e}\star T_{\K}\XA^{1}\star R_{1}^{1}$ &\eqref{EqAX}, \eqref{EqXA}, \eqref{EqBXA}, \eqref{EqYFrak}  & \multirow{2}{*}{left bimodule}   \\\cline{3-4} 
&& Surjctive Calculus: &\multirow{2}{*}{\eqref{EqAX}, \eqref{EqXA}, \eqref{EqBXA}, \eqref{EqSurjBXA}} & \multirow{2}{*}{connections}\\
&&  $A^{e} \star T_{\K}\XA^{1}$ & & \\
 \hline
\multirow{3}{*}{B} & \multirow{3}{*}{$IB\XA^{1}$} & $A^{e}\star T_{\K}\XA^{1}\star R_{1}^{1} \star R_{2}^{1}$  & \eqref{EqAX}, \eqref{EqXA}, \eqref{EqBXA}, \eqref{EqYFrak}, \eqref{EqRelInv1}, \eqref{EqRelInv2} & invertible \\\cline{3-4} 
&& Surjective Calculus:  &\eqref{EqAX}, \eqref{EqXA}, \eqref{EqBXA}, \eqref{EqSurjBXA}, \eqref{EqSurjInvDef} & bimodule \\
 &&$A^{e} \star T_{\K}\XA^{1}\star T_{\K}\mathfrak{Y}^{1}$&\eqref{EqSurjInv1}, \eqref{EqSurjInv2} & connections \\
 \hline
 \multirow{4}{*}{H} & \multirow{4}{*}{$H\XA^{1}$} & \multirow{2}{*}{$A^{e}\star T_{\K}\XA^{1}\star R_{1}^{1} \star R_{2}^{1}$} &\eqref{EqAX}, \eqref{EqXA}, \eqref{EqBXA}, \eqref{EqYFrak}, \eqref{EqRelInv1}, \eqref{EqRelInv2},  & $\Pi$-adaptable \\
  &&&\eqref{EqRelHpf1}, \eqref{EqRelHpf2}& bimodule\\\cline{3-4} 
&& Surjective Calculus:  &\eqref{EqAX}, \eqref{EqXA}, \eqref{EqBXA}, \eqref{EqSurjBXA}, \eqref{EqSurjInvDef} & connections\\
 &&$A^{e} \star T_{\K}\XA^{1}\star T_{\K}\mathfrak{Y}^{1}$&\eqref{EqSurjInv1}, \eqref{EqSurjInv2}, \eqref{EqSurjHpf1}, \eqref{EqSurjHpf2} &(\emph{dependent on }$\Pi$)\\
 
 \hline\hline
  \multirow{2}{*}{-} & \multirow{2}{*}{$\prescript{}{A}{\ct{D}}$} & \multirow{2}{*}{$A\star T_{\K}\XA^{1}$} &\multirow{2}{*}{\eqref{EqAX}, \eqref{EqXA}, \eqref{EqFlat}} & flat left connections \\
 & & & & [Corollary 6.24 \cite{beggs2020quantum}]\\
 \hline
  \multirow{4}{*}{H} & \multirow{4}{*}{$\ct{D}\XA$} & $A^{e}\star T_{\K}\XA^{1}\star R_{1}^{1} \star R_{2}^{1}$ &\eqref{EqAX}, \eqref{EqXA}, \eqref{EqBXA}, \eqref{EqYFrak}, \eqref{EqRelInv1}, \eqref{EqRelInv2},  & flat extendible\\
  &&$ \star R_{1}^{2}\star R_{2}^{2}$ &\eqref{EqRelHpf1}, \eqref{EqRelHpf2}, \eqref{EqFlat}, \eqref{EqFlExtraa}, \eqref{EqExt1},\eqref{EqExt2} &   $\Pi$-adaptable \\\cline{3-4} 
&& Surjective calculus and $\wedge$: &\eqref{EqAX}, \eqref{EqXA}, \eqref{EqBXA},\eqref{EqFlat}, \eqref{EqSurjBXA}, \eqref{EqSurjInvDef}, & bimodule connections  \\
 &&$A^{e} \star T_{\K}\XA^{1}\star T_{\K}\mathfrak{Y}^{1}$& \eqref{EqSurjInv1},\eqref{EqSurjInv2}, \eqref{EqSurjHpf1}, \eqref{EqSurjHpf2}, \eqref{EqSujExt1}, \eqref{EqSujExt2} & (\emph{dependent on }$\Pi$)\\
 \hline
\end{tabular}
\caption{Summary of the results in \cite{ghobadi2020hopf}, in comparison with the work in Section \ref{SSurjective}. Here H/B are short for Hopf algebroid/Bialgebroid}
\end{table}
\end{center}

There is a subtly regarding the connections represented by $\ct{D}\XA$ which is omitted in the above table.  The category of left modules over this algebra is isomorphic to the category of flat extendible $\Pi$-ABCs $(M,\cnc , \sigma )$, satisfying the additional condition 
\begin{equation}\label{EqFlExtraa}
(\wedge \otimes M)[ (\mathrm{id}_{\Omega^{1}}\otimes\sigma)(\cnc \otimes \mathrm{id}_{\Omega^{1}}) +( \mathrm{id}_{\Omega^{1}}\otimes \cnc)\sigma]= (d \otimes \mathrm{id}_{M})\sigma - \sigma_{2}(\mathrm{id}_{M}\otimes d)
\end{equation} 
where this additional condition is hidden in relation \eqref{EqFlExt}. However as noted in \cite{ghobadi2020hopf} and in Lemma 4.12 of \cite{beggs2020quantum}, when the calculus is surjective this condition holds for any flat extendible connection and relation \eqref{EqFlExt} becomes redundant. Since in this work we will focus on the case where the calculus is surjective, we will no longer mention relation \eqref{EqFlExt} and equation \eqref{EqFlExtraa}.

\subsection{Bialgebroids and Hopf algebroids}\label{SPrelimHopf}
In this section we will briefly recall the definition of a bialgebroid and a Hopf algebroid and set our notation. 

An $A^{e}$-\emph{ring} structure on $B$ consists of a $\K$-algebra structure $(\mu ,1_{B})$ on $B$ with an algebra homomorphism $\eta : A^{e}\rightarrow B$. The induced algebra morphisms $A\rightarrow B$ and $A^{\op}\rightarrow B$, are then denoted by $\source$ and $\target$ respectively, and called the \emph{source} and \emph{target} morphisms. Additionally, $B$ is equipped with an $A$-bimodule structure and an $A^{\op}$-bimodule structure, denoted by $\sMs{B}$ and $\tMt{B}$, respectively, while $\Mto{B}$ denotes the left $A^{\op}$-module $\tM{B}$ viewed as a right $A$-module. We will also adapt the notation of \cite{ghobadi2020hopf} and denote the image of elements from $A$ in $B$ by simply writing $a$ instead of $\source (a)$ and similarly write $\ov{a}$ instead of $\target(\ov{a})$ for $\ov{a}\in A^{\op}$, whenever possible.   

An $A$-\emph{coring} structure on $B$ consists of $A$-bimodule maps $\Delta :\sMto{B}\rightarrow \sMto{B}\otimes \sMto{B}$ and $\epsilon :\sMto{B}\rightarrow A$ satisfying
\begin{align} b_{(1)}\otimes(b_{(2)})_{(1)}\otimes(b_{(2)})_{(2)}=&(b_{(1)})_{(1)}\otimes(b_{(1)})_{(2)}\otimes b_{(2)} \label{EqDelAss} \\ 
\epsilon (b_{(1)})b_{(2)} =b&= \ov{\epsilon (b_{(2)})}b_{(1)} \label{EqCoun}\\ 
\Delta (br\ov{s})=& b_{(1)} r\otimes b_{(2)} \ov{s} \label{EqDelrs}  \\ 
\epsilon (br) =& \epsilon (b\ov{r})\label{EqCounrs}
\end{align} 
for any $b\in B$ and $r,s\in A$, where $\Delta (b)= b_{(1)}\otimes b_{(2)}$ is denoted by Sweedler's notation. A \emph{left} $A$-\emph{bialgebroid} structure on $B$ consists of an $A^{e}$-ring structure and an $A$-coring structure on $B$ satisfying  
\begin{align}
(bb')_{(1)}\otimes (bb')_{(2)} &=b_{(1)}b_{(1)}' \otimes b_{(2)}b_{(2)}' , \label{EqDelMul}
\\ \Delta (1_{B})&=1_{B}\otimes 1_{B}\label{Eq:B5}
\\\epsilon (1_{B} )&=1_{A}\label{Eq:B6} 
\\ \epsilon (b b')&= \epsilon ( b\epsilon (b') )\label{EqCounMul}
\end{align}
for any $b,b'\in B$. From here onwards, a bialgebroid will always signify a left bialgebroid and be written as a pair $(A,B)$.

Recall from \cite{schauenburg2000algebras}, a bialgebroid $(A,H)$ is called a \emph{left Hopf algebroid} if the morphism 
\begin{equation}\label{EqHopf}
\beta :\rt{H}\tn_{A}\lt{H} \longrightarrow \Mto{H}\tn_{A} \ls{H}, \quad b\tn_{A} b'\longmapsto b_{(1)}\tn_{A} b_{(2)}b' 
\end{equation}
is invertible. In this case, we denote $ \beta(b\tn_{A} 1_{H}) =b_{(+)}\tn_{A} b_{(-)}$. It follows by definition [Proposition 3.7 \cite{schauenburg2000algebras}] that 
\begin{align}
(\source (a))_{(+)}\tn_{A^{\op}} (\source (a))_{(-)} &= \source (a)\tn_{A^{\op}} 1_{H}\label{eq:s+-}
\\(\target (a))_{(+)}\tn_{A^{\op}} (\target (a))_{(-)}&= 1_{H}\tn_{A^{\op}} \source (a)\label{eq:t+-}
\\(bb')_{(+)}\tn_{A^{\op}} (bb')_{(-)}&= b_{(+)}b'_{(+)}\tn_{A^{\op}} b'_{(-)}b_{(-)}\label{eq:hh'+-}
\end{align}
There's a symmetric notion of right Hopf algebroid where we require the map $\vartheta:\rs{B}\tn_{A} \ls{B}\longrightarrow  \Mto{B}\tn_{A} \ls{B}$ defined by $b\tn b' \mapsto b_{(1)}b'\tn b_{(2)}$ to be invertible with notation $\vartheta^{-1}(b\tn 1)= b_{[+]}\tn b_{[-]}$. Although the examples in consideration, $H\XA^{1}$ and $\ct{D}\XA$, satisfy this additional condition, we will avoid discussing this condition further as the relevant proofs for this condition will be symmetric. Hence, from here onwards, a Hopf algebroid will mean a left Hopf algebroid. 

Finally, recall from \cite{bohm2004hopf} that a bialgebroid $(A,H)$ is said to admit an invertible \emph{antipode} if there exists an invertible anti-algebra morphism $S: H\rightarrow H$ satisfying 
\begin{align}
S (\source(a))= \target (a)&\label{eq:antipode1}
\\ S \left( b_{(1)}\right)_{(1)}b_{(2)}\tn_{A} S \left( b_{(1)}\right)_{(2)}&= 1_{H}\tn_{A}  S \left( b\right)\label{eq:antipode2}
\\S^{-1}\left( b_{(2)}\right)_{(1)}\tn_{A}S^{-1}\left( b_{(2)}\right)_{(2)}b_{(1)} &=S^{-1}\left( b\right) \tn_{A} 1_{H}\label{eq:antipode3}
\end{align}
where \eqref{eq:antipode2} and \eqref{eq:antipode3} are equalities as elements in $\Mto{H}\tn_{A} \ls{H}$. It follows from Proposition 4.2 of \cite{bohm2004hopf} that any such bialgebroid is a Hopf algebroid, with $\beta^{-1}(b\tn_{A} b')= S^{-1}\left(S \left( b \right)_{(2)}\right)\tn_{A} S \left( b \right)_{(1)}b'$. We denote such Hopf algebroids $(A,H)$ which admit an antipode, $S$, as a triple $(A,H, S)$.

We conclude this section with some classical examples of Hopf algebroids: 
\begin{ex}\label{ExFiniteGrpd}\rm Give a finite groupoid $\ct{G}= (\ct{G}_{0},\ct{G}_{1})$, we obtain a Hopf algebroid structure, $ (\K (\ct{G}_{0}),\K(\ct{G}_{1}))$, on the algebra of functions on the groupoid. Here $\K (\ct{G}_{0})$ is the algebra of functions on the set $\ct{G}_{0}$ and is spanned by elements of the form $f_{g}$ corresponding to $g\in  \ct{G}_{0}$ satisfying $f_{g}.f_{g'}= \delta_{g,g'}f_{g}$, while $\K(\ct{G}_{1})$ denotes the algebra of functions on the set of arrows $\ct{G}_{1}$ and generated by elements $f_{ {e}}$ corresponding to arrows $ {e}\in \ct{G}_{1}$, satisfying $f_{ {e_{1}}}. f_{ {e_{2}}}=\delta_{ {e_{1}}, {e_{2}}} f_{ {e_{1}}}$. The Hopf algebroid structure is defined as follows: 
\begin{align*}
\source (f_{p}) = \sum_{ {e}: s( {e})=p} &f_{ {e}},\quad \target (f_{p}) = \sum_{ {e}: t( {e})=p} f_{ {e}},\quad \Delta ( f_{ {e}}) = \sum_{ {e_{1}}, {e_{2}}:\  {e_{2}}. {e_{1}}=  {e}} f_{ {e_{1}}}\tn f_{ {e_{2}}}
\\ \epsilon ( f_{ {e}})& =\sum_{p\in \ct{G}_{0}} \delta_{ {e}, \id_{p}}f_{p}, \quad S\left( f_{ {e}}\right)= f_{ {e}^{-1}} ,\quad  1= \sum_{ {e}\in \ct{G}_{1}} f_{ {e}}
\end{align*}
Notice that the algebras $\K (\ct{G}_{0})$ and $ \K(\ct{G}_{1})$ are well-defined regardless of the finiteness condition, however we need the groupoid to be finite for the structure maps to be well-defined.
\end{ex}
\begin{ex}\label{ExGrpd}\rm Give an arbitrary groupoid $\ct{G}= (\ct{G}_{0},\ct{G}_{1})$ with finite $\ct{G}_{0}$, we obtain a Hopf algebroid structure, $ \K\ct{G}=(\K(\ct{G}_{0}),\K\ct{G}_{1})$, which we call the \emph{groupoid ring} of $\ct{G}$. Here $\K(\ct{G}_{0})$ is as in Example \ref{ExFiniteGrpd}, but $\K\ct{G}_{1}$ is spanned as a vector space by arrows $e\in \ct{G}_{1}$ with multiplication defined as $e\sbt e'= e.e' \delta_{s(e),t(e')}$, where we denote the composition in $\ct{G}_{1}$ by $.$ and the multiplication in $\K\ct{G}_{1}$ by $\sbt$. The Hopf algebroid structure is defined as follows:
\begin{align*}
\source (f_{p}) = \target (f_{p}) =  \id_{p} ,\quad \Delta ( e) = e\tn e, \quad \epsilon ( e) =f_{t(e)}, \quad S\left( e\right)= e^{-1} ,\quad  1= \sum_{p\in \ct{G}_{0}} \id_{p}
\end{align*}
where we need finiteness on $\ct{G}_{0}$ to have a well-defined unit. 
\end{ex}
\begin{ex}\label{ExPresheaf}\rm A Hopf algebroid $(A,H)$ is said to be commutative if $A$ and $H$ are both commutative algebras. It follows by definition that $(A,H)$ admits an antipode defined by $S(h)= \target (\epsilon (h_{(+)})).h_{(-)}$ for $h\in H$, see Section 5.4.1 of \cite{kowalzig2009hopf}. Given any such Hopf algebroids, we have a pre-sheaf of groupoids associated to $(A,H)$, denoted by $\ct{H}: \mathrm{Alg}^{\op}_{\K} \rightarrow \mathrm{Grpd}$, which sends any algebra $C$ to the groupoid $(\ct{H}(C)_{0},\ct{H}(C)_{1})$, where $\ct{H}(C)_{0}=\Hom_{\K - \mathrm{alg}}(A,C)$ and $\ct{H}(C)_{1}=\Hom_{\K - \mathrm{alg}}(H,C)$. We refer to Section 3 of \cite{el2018geometrically}, for additional details.
\end{ex}
\section{Ideals and quotients of Hopf algebroids}\label{SIdeals}
In this section we introduce the appropriate notion of Hopf ideals and quotients for Hopf algebroids, as well as the notion of isotopy quotient. In contrast with the classical setting of Hopf algebras, Hopf algebroids might not have antipodes, which makes the Hopf condition more difficult to check. Additionally, Hopf algebroids are defined with respect to a base algebra and a good notion of ideal should allow for a base change. 

\subsection{Quotients with same base algebra}\label{ssec:Quotients}
Recall from Section 17.14 of~\cite{brzezinski2003corings}, a 2-sided coideal $I$ of an $A$-coring $H$ is an $A$-subbimodule of $\sMto{H}$ which satisfies $ \Delta (I) \subseteq {\rm Im} \Big( \Mto{I} \tn_{A} \ls{H} + \Mto{H} \tn_{A} \ls{I}\Big)\subset \Mto{H} \tn_{A} \ls{H}$ and $\epsilon (I)=0 $, so that $\frac{H}{I}$ becomes an $A$-coring with the projection of $\Delta$ and $\epsilon$. 

\begin{proposition}\label{prop:quotbialg} Let $(A,H )$ be a bialgebroid and $I$ a 2-sided ideal of $H$ and a 2-sided ideal of $(A,H)$. The algebra $\frac{H}{I} $ has an induced $A$-bialgebroid structure such that the natural projection $\pi: H\rightarrow \frac{H}{I}$ becomes a morphism of bialgebroids.
\end{proposition}
\begin{proof} It follows directly that $\pi\eta$ is a $\K$-algebra map and $\frac{H}{I}$ is equipped with induced source and target maps, $\pi s$ and $\pi t$, respectively. By assumption, $\frac{H}{I}$ is also an $A$-coring with the obvious defining maps. What remains to be shown is whether the bialgebroid axioms hold. 

Conditions \eqref{Eq:B5} and \eqref{Eq:B6} follow trivially since $1_{\frac{H}{I}}= 1_{H} +I$ and $\varepsilon (I)=0$. Let $h+I,h'+I\in \frac{H}{I}$. Since $I$ is an $A^{e}$-subbimodule of $B$, 
$$\epsilon' \big( (h+I)(h'+I)\big)= \epsilon' ( hh'+I) = \epsilon' \big( h\epsilon (h')+I  \big)=\epsilon' \big( (h+I)\epsilon (h')  \big)=  \epsilon' \big( (h+I)\epsilon' (h'+I)  \big) $$
Noting that $\Delta' : \frac{H}{I} \rightarrow \sMto{\frac{H}{I}}\tn_{A}\sMto{\frac{H}{I}}$ is defined by $h+I\mapsto h_{(1)}+I \tn_{A} h_{(2)}+I$, it is straightforward to check that $\Delta'$ is an algebra map: 
\begin{align*}
\left( h+I \right)_{(1)}\left( h'+I \right)_{(1)}&\tn_{A} \left( h+I \right)_{(2)}\left( h'+I \right)_{(2)}= \left( h_{(1)}+I \right)\left( h'_{(1)}+I \right)\tn_{A} \left( h_{(2)}+I \right)\left( h'_{(2)}+I \right)
\\&=\left( h_{(1)}h'_{(1)}+I \right)\tn_{A} \left( h_{(2)}h'_{(2)}+I \right) =\left( (hh')_{(1)}+I \right)\tn_{A} \left( (hh')_{(2)}+I \right)
\\&=\left( (h+I)(h'+I) \right)_{(1)}\tn_{A} \left( (h+I)(h'+I) \right)_{(2)}
\qedhere\end{align*}\end{proof}
If $(A,H)$ is a Hopf algebroid, we call a 2-sided ideal and 2-sided coideal $I$ of $H$, a \emph{Hopf ideal} if for any element $h\in I$ we have that
\begin{equation}\label{eq:HpfIdeal}
h_{(+)}\tn_{A^{\op}} h_{(-)} \in \mathrm{Im} \left( \rt{I}\tn_{A^{\op}} \lt{H} + \rt{H}\tn_{A^{\op}} \lt{I} \right) \subseteq \rt{H}\tn_{A^{\op}} \lt{H} \tag{lH}
\end{equation}
holds.
\begin{proposition}\label{prop:quotHpfalg} If $(A,H)$ is a Hopf algebroid and $I$ a Hopf ideal, then the natural bialgebroid $(A,\frac{H}{I})$ is a Hopf algebroid. 
\end{proposition}
\begin{proof} Following from Proposition \ref{prop:quotbialg}, we must prove that $\beta':\rt{\frac{H}{I}}\tn_{A^{\op}}\lt{\frac{H}{I}} \rightarrow \Mto{\frac{H}{I}}\tn_{A} \ls{\frac{H}{I}}$ defined by $\beta' ( b+I\tn_{A^{\op}} b'+I)= b_{(1)}+I\tn_{A} b_{(2)}b'+I$ is invertible. Define $\theta: \Mto{\frac{H}{I}}\tn_{A} \ls{\frac{H}{I}}\rightarrow\rt{\frac{H}{I}}\tn_{A^{\op}}\lt{\frac{H}{I}} $ by $\theta ( b+I\tn_{A} b'+I) = b_{(+)}+I\tn_{A^{\op}} b_{(-)}b' +I$. Observe that because $I$ is a Hopf ideal, $\theta$ is well-defined. It is straightforward to check that $\theta$ is the inverse of $\beta'$. 
\end{proof}
We should note that the right-handed Hopf condition holds if the ideal satisfies 
\begin{equation}\label{eq:RightHpfIdeal}
\vartheta^{-1} ( 1\tn h)\in \mathrm{Im} \left( \rs{I}\tn_{A} \ls{H} + \rs{H}\tn_{A} \ls{I} \right) \subseteq \rs{H}\tn_{A} \ls{H} \tag{rH}
\end{equation}
and a symmetric argument to that of Proposition \ref{prop:quotHpfalg} follows.

We call a 2-sided ideal and 2-sided coideal $I$ of a Hopf algebroid $(A,H, S)$, $S$-\emph{stable} if $S( I )= I$. Observe that if $I$ is $S$-stable then by the description of $\beta^{-1}$ in terms of the antipode, it follows directly that $I$ is a Hopf ideal. 
\begin{proposition}\label{prop:antipode} If $I$ is an $S$-stable ideal of a Hopf algebroid $(A,H, S)$, the quotient Hopf algebroid $(A,\frac{H}{I})$ admits an antipode $S'$, defined by $S'( h+ I)= S( h)+ I$ for $h+I\in \frac{H}{I}$. 
\end{proposition}
\begin{proof} Since $S( I )= I$ and thereby $S^{-1}( I )= I$, it follows that $S'$ is an invertible anti-algebra map and $S'(\pi s )=\pi t$. Conditions \eqref{eq:antipode2} and \eqref{eq:antipode3} then follow from the definition of $S'$ and $\Delta'$ from Proposition \ref{prop:quotbialg}. 
\end{proof}
\subsection{Quotients with base algebra change}\label{ssec:basechng}
If $(A,H)$ is a bialgebroid, we call a pair of 2-sided ideals $J\subset A$ and $I\subset H$ an ideal of $(A,H)$ if
\begin{enumerate}[label=(BQ\arabic*), ref=BQ\arabic*, leftmargin=1.1cm ]
\item \label{Eq:BQ1} $\source (J)\subset I$ and $\target ({J}^{\op})\subset I$, where $J^{\op}$ is the corresponding ideal to $J$ in ${A}^{\op}$
\item \label{Eq:BQ2} $\Delta(I) \subseteq {\rm Im} \Big( \Mto{I} \tn_{A} \ls{H} + \Mto{H} \tn_{A} \ls{I}\Big)\subset \Mto{H} \tn_{A} \ls{H}$
\item \label{Eq:BQ3} $\varepsilon (I)\subseteq J$ 
\end{enumerate}
Thereby, any 2-sided ideal and coideal $I$ as in Proposition \ref{prop:quotbialg}, forms a pair in this sense as $(0,I)$. 

Observe that given an ideal $J\subseteq A$, the subset $J\tn J^{\op}$ becomes an ideal of $A^{e}$ and $\frac{A^{e}}{J\otimes {J}^{\op}}\cong\big( \frac{A}{J}\big)^{e}$. From \eqref{Eq:BQ1}, it follows that $\overline{\eta}:\frac{A^{e}}{J\otimes {J}^{\op}}\rightarrow \frac{H}{I}$ defined by $\overline{\eta} \left(a \tn_{A} a' + J\otimes {J}^{\op} \right)= \eta (a \tn_{A} a') + I $ is an algebra map and we will abuse notation and denote the respective source and target maps by $\source$ and $\target$ again.

The definition of a coideal of a coring which allows for base change first appeared in Definition 1.1.29 of \cite{el2004co}. In particular, \eqref{Eq:BQ2} and \eqref{Eq:BQ3} are equivalent to identities (1) and (2) in Definition 1.1.29 of \cite{el2004co}, while \eqref{Eq:BQ1} implies identity (3). By Proposition 1.1.30 of \cite{el2004co} we conclude that $\frac{H}{I}$ carries a $\frac{A}{J}$-coring structure, in a natural way. 
\begin{proposition}\label{prop:basechngbialg} If $(J,I)$ is an ideal of the bialgebroid $(A,H)$, then the pair $(\frac{A}{J},\frac{H}{I})$ has an induced bialgebroid structure. 
\end{proposition}
\begin{proof} As mentioned above $\frac{H}{I}$ becomes an $A/J$-coring, with structure maps $\varepsilon' (h+I)=\varepsilon(h)+J$ and $\Delta'( h+I)= h_{(1)}+I\tn_{\frac{A}{J}} h_{(1)}+I$, where the latter is well-defined since $\Mto{\frac{H}{I}}\tn_{A}\ls{\frac{H}{I}}\cong \Mto{\frac{H}{I}}\tn_{\frac{A}{J}}\ls{\frac{H}{I}}$. 

Axiom \eqref{EqDelMul} follows trivially as in Proposition \ref{prop:quotbialg}, with the additional observation that $\Mto{\frac{H}{I}}\tn_{A}\ls{\frac{H}{I}}\cong \Mto{\frac{H}{I}}\tn_{\frac{A}{J}}\ls{\frac{H}{I}}$. Conditions \eqref{Eq:B5} and \eqref{Eq:B6} follow trivially since $1_{\frac{H}{I}}= 1_{H} +I$ and $\varepsilon (I)=J$. For \eqref{EqCounMul}, we observe that 
\begin{align*}
\varepsilon'\big( ( h+I)\source \big(\varepsilon'(h'+I)\big) \big) =& \varepsilon'\big( ( h+I)\source\big(\varepsilon(h')+J\big) \big)=  \varepsilon' \big((h+I)\big( \source(\varepsilon(h'))+I\big)\big) = \varepsilon' \big(h \source (\varepsilon(h'))+I\big)
\\=&\varepsilon (hs(\varepsilon(h')))+J= \varepsilon ( hh')+J= \varepsilon' ( (h+I)(h'+I)) 
\end{align*}
for any $h+I,h'+I\in \frac{H}{I}$.
\end{proof}
If $(A,H)$ is a left Hopf algebroid, we call an ideal $(J,I)$ of $(A,H)$ a \emph{Hopf ideal} if for any element $h\in I$, \eqref{eq:HpfIdeal} holds. 
\begin{proposition}\label{prop:HopfBasechng} If $(A,H)$ is a left Hopf algebroid and $(J,I)$ a Hopf ideal, then the bialgebroid $(\frac{A}{J},\frac{H}{I})$ is a left Hopf algebroid. 
\end{proposition}
\begin{proof} It follows from $\source (J),\target (J^{\op})\subset I$ that $\rt{\frac{H}{I}}\tn_{A^{\op}}\lt{\frac{H}{I}} \cong\rt{\frac{H}{I}}\tn_{\frac{A}{J}^{\op}}\lt{\frac{H}{I}} $ and $\Mto{\frac{H}{I}}\tn_{A} \ls{\frac{H}{I}}\cong \Mto{\frac{H}{I}}\tn_{\frac{A}{J}} \ls{\frac{H}{I}}$. Therefore, we can define $\theta: \Mto{\frac{H}{I}}\tn_{\frac{A}{J}} \ls{\frac{H}{I}}\rightarrow \rt{\frac{H}{I}}\tn_{\frac{A}{J}^{\op}}\lt{\frac{H}{I}}$ by $h+I\tn_{\frac{A}{J}} h'+I\mapsto  h_{(-)}+I\tn_{\frac{A}{J}^{\op}} h_{(-)} h'+I$ as done in Proposition \ref{prop:quotHpfalg} and \eqref{Eq:BQ1} guarantees that $\theta'$ will be well-defined. The Hopf algebroid structure of $H$ then implies that $\theta'$ is an inverse of $h+I\tn_{\frac{A}{J}^{\op}} h'+I\mapsto h_{(1)}+I\tn_{\frac{A}{J}} h_{(2)}h'+I$. 
\end{proof}
If $(A,H,S)$ is a Hopf algebroid with an invertible antipode and $(J,I)$ is an ideal of $(A,H)$, where $I$ is $S$-stable, the same argument of Proposition \ref{prop:antipode} shows that the quotient Hopf algebroid $(\frac{A}{J},\frac{H}{I})$ admits an antipode. 

The following examples verify that our notion of Hopf ideal reduce as expected for the classical examples of Hopf algebroids related to groupoids.
\begin{ex}\label{ex:FiniteGrpdIdeals}\rm If $\ct{G}$ is a finite groupoid then Hopf ideals of the Hopf algebroid $\K (\ct{G})$ defined in Example \ref{ExFiniteGrpd} correspond to subgroupoids of $\ct{G}$. It is easy to see that ideals $I_{i}$ of $\K (\ct{G}_{i})$ must be of the form $I_{i}=\mathrm{Span}_{\K} \{f_{x} \}_{x\in S_{i}}$ for some subset $S_{i}\subseteq\ct{G}_{i}$, for $i=0,1$. Following this observation, we see that a pair of ideals $(I_{0},I_{1})$ in $(\K (\ct{G}_{0}),\K(\ct{G}_{1}))$ satisfy \eqref{Eq:BQ1} if $S_{1}$ contains all arrows whose source or target belongs to $ S_{0}$. Condition \eqref{Eq:BQ2} corresponds to the following statement which we denote by (A): if $ {e_{1}}. {e_{2}}\in S_{1}$, then either $  {e_{1}}\in S_{1}$ or $  {e_{2}}\in S_{1}$. Finally, condition \eqref{Eq:BQ3} is equivalent to $S_{1}$ only containing the identity arrows $ {\id_{p}}$ corresponding to points in $p\in S_{0}$. Since 
\begin{align*}
(f_{ {e}})_{(+)}\tn (f_{ {e}} )_{(-)} = \sum_{ {e_{1}}, {e_{2}}:\  {e_{2}}. {e_{1}}=  {e}}f_{ {e_{1}}}\tn f_{ {e_{2}}^{-1}}
\end{align*}
then the Hopf condition \eqref{eq:HpfIdeal} is equivalent to the following statement holding: if $ {e_{1}}. {e_{2}}\in S_{1}$, then either $  {e_{1}}\in S_{1}$ or $  {e_{2}}^{-1}\in S_{1}$. If $(I_{0},I_{1})$ forms a bialgebroid ideal, then \eqref{eq:HpfIdeal} is simply stating that $S_{1}$ is closed under inverse: given $ {e}\in S_{1}$, either $t( {e})\in S_{0}$, in which case $s( {e}^{-1})\in S_{0}$ and by \eqref{Eq:BQ1}, $ {e}^{-1}\in S_{1}$, or $t( {e})\notin S_{0}$, in which case $ {\id_{t( {e})}}\notin S_{1}$ and $\id_{t( {e})}.  {e}=  {e}\in S_{1}$ which implies that $ {e}^{-1}\in S_{1}$. It is now straightforward to check that the mentioned conditions hold if and only if $(\ct{G}_{0}\setminus S_{0}, \ct{G}_{1}\setminus S_{1})$ forms a subgroupoid of $(\ct{G}_{0},\ct{G}_{1})$. Conditions \eqref{Eq:BQ1} ensures that arrows in and $\ct{G}_{1}\setminus S_{1}$ do no interact with points in $S_{0}$, while \eqref{Eq:BQ3} implies that all identity morphisms for points in $\ct{G}_{0}\setminus S_{0}$ appear in $\ct{G}_{1}\setminus S_{1}$. Finally, \eqref{Eq:BQ2} and \eqref{eq:HpfIdeal} ensure that $\ct{G}_{1}\setminus S_{1}$ is closed under composition and inverses, respectively. 
\end{ex}
\begin{ex}\label{ExGrpdIdeals}\rm Give a groupoid $\ct{G}= (\ct{G}_{0},\ct{G}_{1})$ with finite $\ct{G}_{0}$, Hopf ideals of $\K\ct{G}$, as defined in Example \ref{ExGrpd}, correspond to \emph{isolated} subgroupoids of $\ct{G}$ i.e subgroupoids $(\ct{H}_{0},\ct{H}_{1})$, where if $e\in \ct{G}_{1}$ with $s(e)\in \ct{H}_{0}$ or $t(e)\in \ct{H}_{0}$ then $e\in \ct{H}_{1}$. As observed in Example \ref{ex:FiniteGrpdIdeals}, 2-sided ideals $J$ of $\K( \ct{G}_{0})$ correspond to subsets of $\ct{G}_{0}$. Looking at the ideal conditions, we see that \eqref{Eq:BQ1} is equivalent to $\id_{p}\in I$ for $f_{p}\in J$, while \eqref{Eq:BQ3} states that $f_t(e)\in J$ for any $e\in E$. Condition \eqref{Eq:BQ2} holds automatically for any $I$ since $\Delta$ is defined in a trivial manner and so does \ref{eq:HpfIdeal} since $(e)_{(+)}\tn (e )_{(-)} = e\tn e^{-1}\in \rt{I}\tn_{A} \lt{H}$ for $e\in I$ in this case. If $I$ is a 2-sided ideal of $\K\ct{G}_{1}$ and $\id_{p}\in I$, then for any $e\in \ct{G}_{1}$ with $s(e)=p$ or $t(e)=p$, we have $e\in I$. Conversely, if $e\in I$, then $f_{t(e)}\in J$, by \eqref{Eq:BQ3}, and \eqref{Eq:BQ1} implies that $e^{-1}\in I$ holds, since $f_s(e^{-1})\in J$. Consequently, $\epsilon ( e^{-1}) =f_{t(e^{-1})}= f_{s(e)}\in J$. Hence, ideals $(J,I)$ of $(\K (\ct{G}_{0}),\K\ct{G}_{1})$ correspond to isolated subgroupoids 
\end{ex}
\begin{ex}\label{ExPresheafIdeals}\rm Hopf ideals $(I,J)$ of a commutative Hopf algebroid $(A,H)$, as in Example \ref{ExPresheaf}, correspond to representable presheafs of subgroupoids of $\ct{H}$. Note that since in this setting $h_{(+)}\tn h_{(-)}= h_{(1)}\tn S(h_{(2)})$, then \eqref{eq:HpfIdeal} becomes equivalent to the ideal being $S$-stable. Given any ideal $(I,J)$, we denote the corresponding presheaf of $(\frac{A}{J},\frac{H}{I})$ by $\overline{\ct{H}}$. If we denote the natural projection maps $A\twoheadrightarrow \frac{A}{J}$ and $H\twoheadrightarrow \frac{H}{I}$ by $\pi_{0}$ and $\pi_{1}$, the corresponding maps $\Hom_{\K - \mathrm{alg}}(\pi_{n}, C): \overline{\ct{H}}(C)\rightarrow \ct{H}(C)$ become injective maps for any algebra $C$ and $n=0,1$, and commute with the structure maps of the groupoid. Hence, for any algebra $C$, $\overline{\ct{H}}(C)$ forms a subgroupoid of $\ct{H}(C)$. 
\end{ex}
\subsection{Isotopy quotient of a Hopf algebroid}\label{SIso}
In this section we introduce a distinguished quotient of a Hopf algebroid, called its isotropy quotient, and its connection to isotropy groups of groupoids. Recall that for an arbitrary groupoid $(\ct{G}_{0},\ct{G}_{1})$ and $x\in \ct{G}_{0}$, the group of arrows with the same source and target, $x$, is called the isotropy group of the groupoid at the point $x$. The subgroupoid obtained as the disjoint product of these groups is the \emph{isotropy groupoid} of $(\ct{G}_{0},\ct{G}_{1})$.

Given a Hopf algebroid $(A,H)$, one has a distinguished ideal $(J,I)$ of $(A,H)$, where $J= \langle ab-ba\rangle_{a ,b\in\, A}$ and $I=\langle \source (a) -\target (a)\rangle_{a \,\in\, A}$. We call this ideal the \emph{isotropy ideal} of $(A,H)$. 
\begin{lemma}\label{lemma:isoIdeal} The isotopy ideal of any Hopf algebroid is a Hopf ideal. 
\end{lemma} 
\begin{proof} Axiom \eqref{Eq:BQ1} follows since $ \source (a)\target (b)= \target (b) \source (a)$ in $H$, while \eqref{Eq:BQ3} holds trivially. Since $\Delta(  \source (a)- \target (a))=  \source (a)\tn 1_{H} - 1_{H} \tn \target (a)= ( \source (a)-  \target (a))\tn 1_{H} - 1_{H} \tn ( \source (a) - \target (a))$, we conclude that the isotopy ideal is indeed an ideal of $(A,H)$ as defined in the last section. What remains is to show that \eqref{eq:HpfIdeal} holds. Since $\langle  \source (a) - \target (a)\rangle_{a \,\in\, A}$ is the 2-sided ideal generated by elements $ \source (a) - \target (a)$, \eqref{eq:hh'+-} implies that we only need to check \eqref{eq:HpfIdeal} for $h= \source (a) - \target (a)$ for some $a\in A$. The latter follows directly from \eqref{eq:s+-} and \eqref{eq:t+-}: 
\begin{align*}\beta^{-1}\big( \source (a)- \target (a)\tn 1_{H}\big)&=( \source (a))_{(+)}\tn_{A^{\op}} ( \source (a))_{(-)}- ( \target (a))_{(-)}\tn_{A^{\op}} ( \target (a))_{(+)}
\\&=  \source (a)\tn_{A^{\op}} 1_{H}- 1_{H}\tn_{A^{\op}}  \source (a) 
\\&=( \source (a)- \target (a))\tn_{A^{\op}} 1_{H}- 1_{H}\tn_{A^{\op}} ( \source (a)- \target (a)) \in \rt{H}\tn_{A^{\op}} \lt{H}
\end{align*}
and thereby lands in $\mathrm{Im} \left( \rt{I}\tn_{A^{\op}} \lt{H} + \rt{H}\tn_{A^{\op}} \lt{I} \right)$ as required. 
\end{proof}
We call the quotient Hopf algebroid $(\frac{A}{J},\frac{H}{I})$ the \emph{isotopy quotient} of $(A,H)$, noting that $\frac{A}{J}$ is a commutative algebra, sometimes called the \emph{abelianisation} of $A$. We will denote the isotopy quotient of $(A,H)$ by $(\ab{A}, \iso{H})$. If $A$ is commutative to begin with, then $J=0$, and the isotropy Hopf algebroid is again a Hopf algebroid over the same base algebra $A$. If in addition to $A$ being commutative, $\source = \target$ holds, which is the case for the universal enveloping Hopf algebroid of a Lie-Rinehart pair or the groupoid ring of Example \ref{ExGrpd}, then $(\ab{A}, \iso{H})=(A,H)$.

\begin{corollary}\label{corollary:IsoAnti} If $(A,H,S)$ is a Hopf algebroid admitting an invertible antipode, then its isotopy quotient $(\ab{A}, \iso{H})$ also admits an antipode.
\end{corollary}
\begin{proof} This statement follows directly from Proposition \ref{prop:antipode} and the fact that the isotopy ideal is $S$-stable since by definition $S( \source (a)-\target (a))= \target (a) -\source (a)$. 
\end{proof}

By definition, left modules over $\iso{H}$ are precisely left $H$-modules which contain the isotopy ideal $I$ in their annihilator. Since $I$ is generated by terms $\source (a)-\target (a)$ and the source and target define the $A$-bimodule structure of the $H$-module, we see that $\iso{H}$-modules are exactly the $H$-modules where the left and right $A$-actions agree i.e. symmetric $A$-modules. By viewing symmetric $A$-bimodules as left $\ab{A}$-modules, we have the following observation:
\begin{thm} Given a bialgebroid $(A,H)$, we have a commutative diagram of functors defined by restriction of scalars:
\begin{equation}
\xymatrix{ \lmod{H}\ar[r]^{\ct{G}}  &\bim{A} \\ \lmod{\iso{H}} \ar[u]^{\ct{F}} \ar[r]& \lmod{\ab{A}}\ar[u]_{\ct{H}}  } 
\end{equation}
where $\ct{F}\big(\lmod{\iso{H}})$ is equal to the full subcategory of modules whose image under $\ct{G}$ falls in $\ct{H}\big( \lmod{\ab{A}}\big)$. 
\end{thm}
The following examples show how our construction reduces to the usual notion of isotopy sub-groupoid, when applied to Hopf algebroids related to groupoids.
\begin{ex}\label{Ex:IsoFinGrpd}\rm Let $(\ct{G}_{0},\ct{G}_{1})$ be a finite groupoid and $\K (\ct{G})$ the corresponding Hopf algebroid as in Example \ref{ExFiniteGrpd}. Since $\K (\ct{G}_{0})$ is commutative, then $\ab{\K (\ct{G}_{0})}= \K (\ct{G}_{0})$. Note that for any $p\in \ct{G}_{0}$, we have that
\begin{equation*}
  \sum_{ {e}: s( {e})=p} f_{ {e}}= \source (f_{p}) \target (f_{p}) = \sum_{ {e}: t( {e})=p} f_{ {e}}
\end{equation*} 
holds in $\iso{\K (\ct{G}_{1})}$. Multiplying the equation by $f_{e}$ for an arbitrary $e\in \ct{G}_{1}$ on both sides, we have that $\delta_{s(e),p}f_{e}= \delta_{t(e),p}f_{e}$. Consequently, in $\iso{\K (\ct{G}_{1})}$ we have $f_{e}= \delta_{s(e),t(e)}f_{e}$ and $(\K (\ct{G}_{0}), \iso{\K (\ct{G}_{1})})$ is isomorphic to the Hopf algebroid of functions on the isotropy subgroupoid $\ct{G}'$ of $\ct{G}$, $(\K (\ct{G}'_{0}), \K (\ct{G}'_{1}))$.
while the isotropy ideal 
\end{ex} 
\begin{ex}\label{Ex:IsoGrpd}\rm Let $(A,H)$ be a commutative Hopf algebroid and $\ct{H}$ its associated pre-sheaf of groupoids, as in Example \ref{ExPresheafIdeals}. Let $\ct{H}^{\mathrm{iso}}$ denote the associated pre-sheaf of groupoids of the isotopy Hopf algebroid of $(A,H)$. By definition, for an arbitrary $\K$-algebra $C$, we have an injective morphism of groupoids $\Hom_{\K - \mathrm{alg}}(\pi,C): \ct{H}^{\mathrm{iso}}(C) \hookrightarrow \ct{H}(C)$, where $\pi: H\rightarrow H/\langle s(a) -t(a)\rangle_{a \,\in\, A}$ denotes the natural projection as before. It is straightforward to see that $\ct{H}^{\mathrm{iso}}(C)$ is the subgroupoid of $\ct{H}(C)$ with the same base points $\ct{H}^{\mathrm{iso}}(C)_{0}= \ct{H}(C)_{0}=\Hom_{\K - \mathrm{alg}}(A,C)$ and arrows which have the same source and target: 
$$\ct{H}^{\mathrm{iso}}(C)_{1}\longleftrightarrow\lbrace  f\in \ct{H}(C)_{1}\mid f(\source (a))=f(\target (a)) \ \forall a\in A\rbrace = \lbrace  f\in \ct{H}(C)_{1} \mid \source^{*}(f)=\target ^{*}(f)\rbrace $$
Hence, $\ct{H}^{\mathrm{iso}}(C)$ is precisely the isotopy groupoid of $\ct{H}(C)$. The notion of isotropy groups was already extended to this setting in Section 5.1 of \cite{el2018geometrically}.
\end{ex}

\section{The case of surjective calculi}\label{SSurjective}
\subsection{Algebra structures}\label{SSurjAlg}
In this section, we develop our constructions in \cite{ghobadi2020hopf} with the further assumption that the calculus is surjective i.e. $\Omega^{1}$ is spanned by elements of the form $a.db$ for $a,b\in A$. If the calculus is surjective, then generators of the form $(x,\omega)\in R_{1}^{1}$ in $B\XA^{1}$ are spanned by generators from $A^{e}$ and $T_{\K}\XA^{1}$: if $\omega= \sum_{m} a_{m}db_{m}$, then by \eqref{EqBXA} it follows that $(x,\omega)=\sum_{m} [x,\ov{b_{m}}]\sbt\ov{a_{m}}$ holds in $B\XA^{1}$. Hence, $B\XA^{1}$ can be viewed as quotient of $A^{e}\star T_{\K}\XA^{1}$ by suitable relations. Consequently, for a surjective calculus we can describe $IB\XA^{1}$ and $H\XA^{1}$ as quotients of $T\XA^{1}_{\sbt}\star T\mathfrak{Y}^{1}_{\sbt}$, where $T\XA^{1}_{\bullet}$ and $T\mathfrak{Y}^{1}_{\sbt}$ are the algebras representing the categories of left and right connections. 

Recall that if the calculus is surjective, then $\Omega^{1}$ is a quotient of the universal calculus on $A$ [Proposition 1.5 \cite{beggs2020quantum}]. The universal calculus on an algebra $A$ is defined as $\Omega^{1}_{\mathrm{uni}}:=\ker(.)\subset A\tn_{\K}A$ with its differential defined by $d_{\mathrm{uni}}a= a\tn_{\K} 1 - 1\tn_{\K}a $. Any other surjective calculus on $A$ will be a quotient of $\Omega^{1}$ by a sub-bimodule $\ct{N}$ [Proposition 1.5 \cite{beggs2020quantum}]. The calculus $\Omega^{1}_{\mathrm{uni}}$ itself is naturally surjective since for any $\sum_{i}a_{i}\tn_{\K} b_{i}\in \Omega^{1}_{\mathrm{uni}}$, $\sum_{i}a_{i}.b_{i}=0$, so that $\sum_{i} a_{i}\tn_{\K} b_{i} =\sum_{i} (d_{\mathrm{uni}}a_{i}). b_{i}$. 

In our notation, we describe our algebras as quotients of coproducts (free products) of some generating algebras by ideals which are defined by a series of relations. For example $T\XA^{1}_{\sbt}$ is defined as the quotient of $A\star T_{\K}\XA^{1}$, by the \emph{ideal} generated by relations \eqref{EqAX} and \eqref{EqXA}. In order to describe an algebra morphism with such a domain, we must define the image of the generators under the morphism and check that the image of these relations hold in the target algebra. Meaning, for a morphism with domain $T\XA^{1}_{\sbt}$, we first define an algebra morphism from $A\star T\XA^{1}_{\sbt}$ and then check that the image of relations \eqref{EqAX} and \eqref{EqXA} hold in the target algebra. If this is the case, it follows that the algebra morphism must factorise through $T\XA^{1}_{\sbt}$. In what follows we will use this subtlety without mentioning. 

\begin{prop}\label{PSurjBXA} If $\Omega^{1}=\Omega^{1}_{\mathrm{uni}}/ \ct{N}$, then $B\XA^{1}$ is isomorphic, as an algebra, to the quotient of $A^{\op} \star T\XA^{1}_{\bullet}$ by relations $a\sbt \ov{b}= \ov{b}\sbt a$ for $a,b\in A$ and 
\begin{equation}\label{EqSurjBXA}
\sum_{m} \ov{s_{m}}\sbt x \sbt \ov{r_{m}} = 0
\end{equation}
for any $x\in \XA^{1}$ and $\sum_{m} r_{m}\tn_{\K} s_{m}\in \ct{N}$. 
\end{prop}
\begin{proof} Let us denote the quotient of $A^{\op} \star T\XA^{1}_{\bullet}$ by the mentioned relations, by $B$. We can define an algebra morphism $\kappa : B\XA^{1}\rightarrow B$ by sending generators from $A^{e}$ and $\XA^{1}$ to the corresponding generators in $B$ and defining $\kappa \big((x,\sum_{i}a_{i}db_{i}) \big) = \sum_{i}[x,\ov{b_{i}}]\sbt \ov{a_{i}}$ for any $\sum_{i}a_{i}db_{i}$. By relation \eqref{EqSurjBXA}, $\kappa$ is well defined since if $\sum_{i}r_{i}ds_{i}= 0$ or equivalently $\sum_{i} r_{i}\tn_{\K}s_{i}\in \ct{N}$, then $\kappa \big((x,\sum_{i}r_{i}ds_{i}) \big)=  \sum_{i} x\sbt \ov{s_{i}}\sbt \ov{r_{i}}-\ov{s_{i}}\sbt x \sbt \ov{r_{i}}= 0$. Relation \eqref{EqBXA} is clearly respected by the morphism and we must only check that \eqref{EqYFrak} holds, for $n=1$, under the morphism: Noting that $(d_{\mathrm{uni}}b).a = d_{\mathrm{uni}}(ba)-b.d_{\mathrm{uni}}a$ in $\Omega^{1}_{\mathrm{uni}}$, we observe that 
\begin{align*} 
\kappa \left( \ov{a}\sbt \big(x,\sum_{i}a_{i}db_{i}\big)\sbt \ov{b}\right)=&\sum_{i} \ov{a}\sbt [x,\ov{b_{i}}]\sbt \ov{a_{i}}\sbt \ov{b}= \sum_{i} \ov{a}\sbt [x,\ov{b_{i}}]\sbt \ov{b.a_{i}}
\\ = & \sum_{i}  [x,\ov{b_{i}a}]\sbt \ov{b.a_{i}}- \sum_{i}  [x,\ov{a}]\ov{b_{i}}\sbt \ov{b.a_{i}}= \kappa \left(  (x,\sum_{i}b.a_{i}db_{i}.a)\right)
\end{align*}
We can define an inverse map for $\kappa$ in a completely symmetric manner, demonstrating that the two presentations are isomorphic.
\end{proof}
Recall that there is a distinction between left and right bimodule connection where a bimodule is considered to either satisfy the Leibniz rule with respect to its left or right $A$-action. We denote the algebra whose category of left modules is isomorphic to the category of right connections, $\ct{E}_{A}$, by  $T_{\K}\mathfrak{Y}^{1}_{\sbt}$. We require $\Omega^{1}$ to be left fgp with its right dual bimodule denoted by $\mathfrak{Y}^{1}$. In this case $T_{\K}\mathfrak{Y}^{1}_{\sbt}$ is define as the quotient of $A^{\op}\star T_{\K}\mathfrak{Y}^{1}$ by relations 
\begin{equation}\label{EqTXop}
\ov{a}\sbt \ov{y}= \ov{ya}, \hspace{0.6cm} \ov{y}\sbt \ov{a} = \ov{ay}+ \ov{\underline{\mathrm{ev}}(da \otimes y)} 
\end{equation} 
where we denote generators of $\fr{Y}^{1}$ by $\ov{y}$ and $\ov{a}\in A^{\op}$. It is well-known that any left bimodule connection $(M,\cnc, \sigma )$ with an invertible $\Omega^{1}$-intertwining, $\sigma$, induces a right bimodule connection $(M,\sigma^{-1}\cnc, \sigma^{-1})$ on $M$ and vice-versa. As observed in Remark 4.4 of \cite{ghobadi2020hopf}, this is because $T\mathfrak{Y}^{1}_{\sbt}$ is a subalgebra of $IB\XA^{1}$: generators from $\mathfrak{Y}^{1}$ are embedded as $\ov{y}\mapsto  (\omega_{i},y)\sbt x_{i}$ and genrators from $A^{\op}$ embed trivially. 

If the calculus is surjective, then generators of the form $(\omega,y)\in R_{2}^{1}$ can be written in terms of commutators between generators from $\mathfrak{Y}$ and $A$. Hence, we can extend the idea of Proposition \ref{PSurjBXA} and fully describe the algebra $IB\XA^{1}$ with generators from $A^{e}$, $\XA^{1}$ and $\mathfrak{Y}^{1}$. Before we state this description, note that when $\Omega^{1}$ is surjective, then $\cvl (1) = \sum_{i}\sum_{l=1}^{l_{i}}(db_{i,l})a_{i,l}\tn x_{i}= \sum_{i}\sum_{l=1}^{l_{i}}db_{i,l}\tn a_{i,l} x_{i}$. Hence, by re-labelling $x_{i}$ and renumbering the indexes, we can always assume that $\cvl (1)=\sum_{i} d b_{i}\tn x_{i}$ and by a symmetric argument that $\cvr (1)= \sum_{j} y_{j}\tn dc_{j}$.  
\begin{prop}\label{PSurjIBXA} If the calculus is surjective, $\Omega^{1}=\Omega^{1}_{\mathrm{uni}}/ \ct{N}$, and $\XA^{1}$ and $\mathfrak{Y}^{1}$ denote the left and right dual bimodules of $\Omega^{1}$ as before, then as an algebra $IB\XA^{1}$ is isomorphic to the quotient of the algebra $B\XA^{1}\star T_{\K}\mathfrak{Y}$, where $B\XA^{1}$ is described as in Proposition \ref{PSurjBXA}, by relations \eqref{EqTXop} and 
\begin{align}
\sum_{m} r_{m}\sbt \ov{y} \sbt s_{m}& = 0  \label{EqSurjInvDef}
\\  [\ov{y},b_{i}]\sbt x_{i}=&\ov{y} \label{EqSurjInv1}
\\  [x,\ov{c_{j}}]\sbt \ov{y_{j}}=&x\label{EqSurjInv2}
\end{align}
where $\sum_{m} r_{m}\tn_{\K} s_{m}\in \ct{N}$.
\end{prop}
\begin{proof} In an analogues manner to the proof of Proposition \ref{PSurjBXA}, we can construct an algebra morphism $\kappa$ from $IB\XA^{1}$ to the algebra described above, which we denote $IB$. We define $\kappa$ on the additional generators of $IB\XA^{1}$ by $(\sum_{i}a_{i}db_{i}, y ) \mapsto a_{i}\sbt[\ov{y},b_{i}]$. In a symmetric argument to the proof of Proposition \ref{PSurjBXA}, it follows that $\kappa$ is a well defined algebra morphism and respects relations \eqref{EqYFrak}, while relations \eqref{EqRelInv1} and \eqref{EqRelInv2} follow directly from \eqref{EqSurjInv1} and \eqref{EqSurjInv2}: 
\begin{align*} \kappa \left( (db_{i},y)\sbt (x_{i},da) \right) &=[\ov{y},b_{i}]\sbt [x_{i},\ov{a}]=[\ov{y},b_{i}]\sbt x_{i}\sbt \ov{a}-[\ov{ay},b_{i}]\sbt x_{i}- [\ov{\evr(da\tn y)},b_{i}]\sbt x_{i}= \ov{y}\sbt \ov{a} - \ov{ay}
\\ & =\ov{\underline{\mathrm{ev}}(da \tn y)}= \kappa \left( \ov{\underline{\mathrm{ev}}(da \tn y)} \right)
\\ \kappa \left( (x,dc_{j})\sbt (da,y_{j}) \right)&= [x,\ov{c_{j}}]\sbt [\ov{y_{j}},a]=[x,\ov{c_{j}}]\sbt \ov{y_{j}} \sbt a- [x\sbt a,\ov{c_{j}}]\sbt \ov{y_{j}}=  x\sbt a - xa- [ \evl ( x\tn da) , \ov{c_{j}}]\sbt\ov{y_{j}} 
\\ &=\mathrm{ev}(x\otimes da)= \kappa\left(\mathrm{ev}(x\otimes da) \right) 
\end{align*}
We can then define an inverse for $\kappa$ by sending additional generators $\ov{y}$ to $ (db_{i},y)\sbt x_{i}$ in $IB\XA^{1}$. The only non-trivial fact to check is whether $\kappa^{-1}$ respects \eqref{EqSurjInv1} and \eqref{EqSurjInv2}, where the latter follows from relation \eqref{EqRelInv2}: 
\begin{align*}
 \kappa^{-1}\left([\ov{y},b_{i}]\sbt x_{i}\right)=&\left[ (db_{l},y)\sbt x_{l},b_{i}\right]\sbt x_{i}=  (db_{l},y)\sbt x_{l} \sbt b_{i}\sbt x_{i}-(b_{i}db_{l},y)\sbt x_{l}\sbt x_{i}
\\=& (db_{l},y)\sbt x_{l} \sbt b_{i}\sbt x_{i}-(db_{l},y)\sbt (x_{l}b_{i}) \sbt x_{i}=(db_{l},y)\sbt \evl (x_{l} \tn db_{i})\sbt x_{i}= (db_{l},y)\sbt x_{l}=\kappa^{-1}(\ov{y})
\\\kappa^{-1}\left( [x,\ov{c_{j}}]\sbt \ov{y_{j}}\right) =& (x,dc_{j})\sbt(db_{i},y_{j})\sbt x_{i}=\evl ( x\tn db_{i}) \sbt x_{i}= x=\kappa^{-1}(x) 
\end{align*}
Hence, $\kappa$ and $\kappa^{-1}$ are both algebra morphisms and clearly inverses.\end{proof}
\begin{corollary}\label{CSurjHXA} If $\Omega^{1}$ is pivotal as a bimodule and $\XA^{1}=\fr{Y}^{1}$, then $H\XA^{1}$ is a quotient of $IB\XA^{1}$ by the additional relations 
\begin{align}
[y_{j},\ov{a}]\sbt [\ov{x},c_{j}]  =&\ov{\evl (x\tn da)} \label{EqSurjHpf1}
\\  [\ov{x_{i}},a]\sbt [ x, \ov{b_{i}}]=&\evr (da\tn x) \label{EqSurjHpf2}
\end{align}
where $IB\XA^{1}$ is presented as in Proposition \ref{PSurjIBXA}.
\end{corollary}
\begin{proof} The proof follows exactly as in the proof of Proposition \ref{PSurjIBXA}, where we relations \eqref{EqSurjHpf1} and \eqref{EqSurjHpf2} are precisely the images of relations \eqref{EqRelHpf1} and \eqref{EqRelHpf2}, respectively, under $\kappa$.
\end{proof}

One might be tempted to look at the quotient of $IB\XA^{1}$ by relations $y_{j}\sbt [\ov{x},c_{j}]  =\ov{x}$ in a symmetric manner to relations \eqref{EqRelInv1} and \eqref{EqRelInv2}. However, such a relation would imply that $[ y_{j},\ov{a}]\sbt [\ov{x},c_{j}] = \ov{\evr (da\tn x)}$:
\begin{align*}
 \ov{xa}-\ov{ax}-\ov{\evr ( da\tn x)}&=\ov{a}\sbt \ov{x}-\ov{x}\sbt \ov{a}=\ov{a}\sbt y_{j}\sbt [\ov{x},c_{j}]- y_{j}\sbt [\ov{x},c_{j}] \sbt \ov{a} 
 \\&= \ov{a}\sbt y_{j}\sbt [\ov{x},c_{j}]-y_{j}\sbt \ov{a}\sbt [\ov{x},c_{j}]- y_{j}\sbt \left[[\ov{x},\ov{a}],c_{j} \right]
\\ &= [\ov{a}, y_{j}]\sbt [\ov{x},c_{j}] + y_{j}\sbt \left[\ov{xa}-\ov{ax}-\ov{\evr (da\tn x)},c_{j} \right]
\\&= [\ov{a}, y_{j}]\sbt [\ov{x},c_{j}] +\ov{xa}-\ov{ax}
\end{align*}
This is a very unnatural implication since for arbitrary $b\in A$ we have
\begin{align*}
\ov{\evr (b.da\tn x)} =  \ov{\evr (da\tn x)}\sbt \ov{b}, \quad [ y_{j},\ov{a}]\sbt [\ov{x},c_{j}]\sbt \ov{b}= [ y_{j},\ov{a}]\sbt [\ov{bx},c_{j}]
\end{align*} 
Hence, the mentioned equality would imply that $\ov{\evr (b.da\tn x)}= \ov{\evr (da\tn bx)}$. Even in the simple case of the digraph calculus, Section \ref{SGrph}, for any edge $\arr{p}{q} \in E $, this relation would imply that $0=\ov{\evr (df_{q}\tn f_{p}.e_{\arl{p}{q}})} = \ov{\evr (f_{p}df_{q}\tn e_{\arl{p}{q}})}= \ov{f_{p}}$.

\subsection{Surjective $\wedge: \Omega^{1}\tn \Omega^{1}\rightarrow \Omega^{2}$}
In order to describe the algebra representing extendible (invertible) bimodule connections, we are forced to add generators from $R_{1}^{2}$ (resp. $R_{2}^{2}$). In this section, we assume that $\wedge : \Omega^{1}\tn \Omega^{1}\rightarrow \Omega^{2}$ is surjective. This assumption makes these additional generators redundant and we show how relations \eqref{EqExt1} and \eqref{EqExt2} can be replaced. Note that any \emph{exterior algebra} i.e. a DGA it is generated by $A$ and $dA$, satisfies the condition that both $\wedge : \Omega^{1}\tn \Omega^{1}\rightarrow \Omega^{2}$ and the calculus are surjective. And since we do not require the additional data of $\Omega^{n}$ for $n\geq 3$, we can assume that the DGA under consideration is an exterior algebra without loss of generality.
\begin{prop}\label{PExt} If $\wedge: \Omega^{1}\tn \Omega^{1}\rightarrow \Omega^{2}$ is surjective, then the algebra representing the categories of extendible left/invertible/adaptable bimodule connections, will be isomorphic to the quotient of $B\XA^{1}$/$IB\XA^{1}$/$H\XA^{1}$ by relations
\begin{align}
\evl^{2}(x^{2}\otimes \omega_{i}\wedge \omega_{j}) \sbt (x_{j},\rho)\sbt(x_{i},\omega)=0\label{EqExtExt1}
\end{align}
and in the case of $IB\XA^{1}$/$H\XA^{1}$, relations
\begin{align}
 \ov{\evl^{2}(\rho_{i}\wedge \rho_{j}\otimes x^{2})}\sbt(\omega ,y_{i})\sbt(\rho, y_{j} )=0\label{EqExtExt2}
\end{align}
where $\omega\wedge \rho\in \mathrm{ker} (\wedge)$.
\end{prop}
\begin{proof} The proof of this statement is analogues to that of Proposition \ref{PSurjIBXA} and boils down to the fact that we can define a $\kappa$ which sends the additional generators $(x^{2}, \omega\wedge \rho)$ from $R_{1}^{2}$ to $\evl^{2}( x^{2}\otimes \omega_{i}\wedge \omega_{j}) \sbt (x_{j},\rho)\sbt(x_{i},\omega)$. Relation \eqref{EqExtExt1} then insures that $\kappa$ is well defined. Similarly relations \eqref{EqExt1} and \eqref{EqExt2}, allow for $\kappa^{-1}$ to be defined. 
\end{proof}
If we use the description of $B\XA^{1}$/$IB\XA^{1}$/$H\XA^{1}$ as in Proposition \ref{PSurjIBXA}, then relations \eqref{EqExtExt1} and \eqref{EqExtExt2} can be rewritten as 
\begin{align}
\mathfrak{ev}(x^{2}\otimes \omega_{i}\wedge \omega_{j}) \sbt [x_{j},\ov{b_{m}}] \sbt [x_{i},\ov{a_{m}}]\sbt\ov{r_{m}}=0 \label{EqSujExt1}
\\\ov{\underline{\mathfrak{ev}}(\rho_{i}\wedge \rho_{j}\otimes x^{2})}r_{m} \sbt [\ov{y_{i}},a_{m}]\sbt [\ov{ y_{j}}, b_{m}]=0\label{EqSujExt2}
\end{align}
respectively, where $\sum_{m}r_{m}da_{m}\wedge db_{m}\in \mathrm{ker}(\wedge)$. 
\begin{corollary}\label{CDXA} If $(A,\oplus_{0 \leq i}\Omega^{n},d,\wedge)$ is a DGA over a surjective calculus with a pivotal structure and $\wedge: \Omega^{1}\tn \Omega^{1}\rightarrow \Omega^{2}$ is surjective, then the algebra representing flat extendible ABCs, $\ct{D} \XA$, is the quotient of $H\XA^{1} $ by relations \eqref{EqSujExt1}, \eqref{EqSujExt2} and \eqref{EqFlat}. 
\end{corollary}
\begin{proof} The proof follows trivially from previous results and the redundancy of relation \eqref{EqFlExt}, which was observed before Remark 5.1 in \cite{ghobadi2020hopf}.
\end{proof}
If the DGA is the maximal prolongation of a first order differential calculus, then in Lemma 4.14 of \cite{beggs2020quantum} it is stated that any flat left (right) bimodule connection is automatically extendible. Hence, in this setting the relations \eqref{EqSujExt1} and \eqref{EqSujExt2} of Proposition \ref{PExt}, would become redundant in $\ct{D}\XA$. 

In Lemma 5.3 of \cite{ghobadi2020hopf}, we have described several additional relations which follow immediately from our description of $\ct{D}\XA$. In particular, by describing $H\XA^{1}$ as a quotient of $T\XA_{\sbt}^{1}\star T\mathfrak{Y}_{\sbt}^{1}$, relation (58) in \cite{ghobadi2020hopf} translates to 
\begin{equation}
\ov{\underline{\mathfrak{ev}}(d\rho_{j}\otimes x^{2}) }\sbt\ov{y_{j}} + \ov{\underline{\mathfrak{ev}}(\rho_{m}\wedge \rho_{n}\otimes x^{2})}\sbt \ov{y_{m}}\sbt \ov{y_{n}}=0\label{EqFlatRght}
\end{equation}
The quotient of $T\mathfrak{Y}^{1}_{\sbt}$ by this exact relation would be the algebra representing the category of flat right connections, denoted by $\ct{D}_{A}$. Hence, $\ct{D}\XA$ can be viewed as a quotient of $ \prescript{}{A}{\ct{D}}\star \ct{D}_{A}$ rather than $T\XA_{\sbt}^{1}\star T\mathfrak{Y}_{\sbt}^{1}$.

\begin{rmk}\label{RRightFlat}\rm Right connections $(M,\cnc^{r})$ can be viewed as left $T\mathfrak{Y}^{1}_{\sbt}$-modules by defining $\cnc^{r}(m) = y_{j}\triangleright m \tn \rho_{j}$. From this point of view the relation \eqref{EqFlatRght} translates to $(\cnc^{r}\wedge \id_{\Omega^{1}}+ \id_{M}\tn d)\cnc^{r} =0$. We must mention that this is the correct notion of flatness for a right connection. Although a left connection $(N, \cnc^{l})$ is flat if $(\id_{\Omega^{1}}\wedge \cnc^{l} - d\tn \id_{M})\cnc^{l} =0$, even for a classical flat connection on a smooth manifold, where $\sigma =\flip$, the induced right connection $\flip\cnc^{l}$ would satisfy the mentioned right flatness condition. In the classical case this change of sign, $\pm$, when changing from left to right is due to $\Omega^{2}$ being the space of antisymmetric forms, $\bigwedge^{2}(\Omega^{1})$, while our work above shows that this change of sign from left to right is independent of $\Omega^{2}$.
\end{rmk}

\subsection{Bialgebroid and Hopf algebroid structures}\label{SBialSurj}
Finally, we will recall the bialgebroid and Hopf algebroids structure of these algebras, with regards to their new generators when the calculus is surjective. We will use the notation $\equiv$ to move between the presentations of these algebras in \cite{ghobadi2020hopf} and the new set of generators described in the previous two sections.

In \cite{ghobadi2020hopf}, we defined the bialgebroid structures of $B\XA^{1}$/$IB\XA^{1}$/$H\XA^{1}$ on the generators of the algebras and showed that we can extend them to the rest of the algebras, multiplicatively. Hence, let us first recall the action of comultiplication and counit on the generators of $B\XA^{1}$ i.e. $ A^{e}$ and $\mathfrak{X}^{1}$. Recall from Proposition \ref{PSurjBXA} that generators $(x,db)$ from $R_{1}^{1}$ are realised as $[x,\ov{b}]$. Hence by Lemma 3.4 of \cite{ghobadi2020hopf}, we define the comultiplication $\Delta :\sMt{B\XA^{1}} \rightarrow \sMt{B\XA^{1}} \otimes \sMt{B\XA^{1}} $ and $\epsilon :\sMt{B\XA^{1}} \rightarrow A$ by 
\begin{align}
&\Delta (a\ov{b})= a \tn \ov{b}, \hspace{2cm} \epsilon ( a\ov{b}) = ab
\\ \Delta (x) = & x\otimes 1 +  [x,\ov{b_{i}}]\otimes x_{i}, \hspace{1.35cm} \epsilon (x)=0
\end{align}
Here when we say the bialgebroid structure is extended multiplicatively, we mean that that $\Delta (m\sbt n)= m_{(1)}\sbt n_{(1)}\otimes m_{(2)}\sbt n_{(2)}$ and $\epsilon (m\sbt n) = \epsilon (m\sbt  \epsilon (n))$ for arbitrary elements $m,n$ in $B\XA^{1}$. 

As described above, for a surjective calculus we can describe $IB\XA^{1}$ and $H\XA^{1}$ as quotients of $T\XA^{1}_{\sbt}\star T\fr{Y}^{1}_{\sbt}$. To describe the action of the comultiplication and counit on generators $\ov{y}\in \fr{Y}^{1}$, we use the fact that in our description $\ov{y}\equiv (\omega_{i},y_{j})\sbt x_{i}$ and $\Delta ((\omega, y ))= ( \omega ,y_{j})\tn ( \rho_{j},y)$ and $\epsilon ((\omega, y ))=\evr (\omega\tn y ) $ for $(\omega, y )\in R^{1}_{2}$, by Lemma 4.2 of \cite{ghobadi2020hopf}. By relation \eqref{EqRelInv1}, we observe that
\begin{align*}
\Delta (\ov{y}) &\equiv\Delta\big((\omega_{i},y)\sbt x_{i}\big) = (\omega_{i},y_{j}) \sbt x_{i}\tn ( \rho_{j},y)+ (\omega_{i},y_{j}) \sbt (x_{i}, \omega_{l}) \tn ( \rho_{j},y) \sbt x_{l}
\\ &= (\omega_{i},y_{j}) \sbt x_{i}\tn ( \rho_{j},y)+ \ov{\evr (\omega_{l}\tn y_{j})} \tn ( \rho_{j},y) \sbt x_{l} \equiv  \ov{y_{j}}\tn [\ov{y},c_{j} ] +1\tn \ov{y}
\end{align*}
Similarly, we note that 
$$\epsilon (\ov{y}) \equiv\epsilon\big((\omega_{i},y_{j})\sbt x_{i}\big)=\epsilon\big((\omega_{i},y_{j})\sbt \epsilon (x_{i})\big)=0 $$
Hence, we have fully described the bialgebroid structures on $IB\XA^{1}$ and $H\XA^{1}$. For a pivotal calculus $\Omega^{1}$, $H\XA^{1}$ is a Hopf algebroid and we can recall its Hopf algebroid structure from Theorem 4.12 of \cite{ghobadi2020hopf}: 
\begin{align*}
x_{(+)} \tn_{A^{\op}} x_{(-)} =& x \tn_{A^{\op}} 1 -(x,\rho_{j}) \tn_{A^{\op}} (\omega_{i},y_{j})\sbt x_{i} \equiv x \tn_{A^{\op}} 1 - \ov{ }[x,\ov{c_{j}}] \tn_{A^{\op}} \ov{y_{j}}
\end{align*}
We can write out $\ov{y}_{(+)}\tn \ov{y}_{(-)}$ for generators $\ov{y}\in \fr{Y}^{1}$ through the identity $\ov{y}\equiv (\omega_{i},y)\sbt x_{i}$ and equations (42) and (43) of \cite{ghobadi2020hopf}:
\begin{align*}
\ov{y}_{(+)} \tn_{A^{\op}} \ov{y}_{(-)} \equiv&   (\omega_{i},y)_{(+)}\sbt(x_{i})_{(+)} \tn_{A^{\op}} (x_{i})_{(-)}\sbt (\omega_{i},y)_{(-)} 
\\= & (\omega_{i},x_{k}) \sbt  x_{i} \tn_{A^{\op}} (y, \omega_{k}) - (\omega_{i},x_{k}) \sbt (x_{i},\rho_{j}) \tn_{A^{\op}}  (\omega_{l},y_{j})\sbt x_{l}\sbt (y, \omega_{k})
\\= & (\omega_{i},x_{k}) \sbt  x_{i} \tn_{A^{\op}} (y, \omega_{k}) - \ov{\evr ( \rho_{j}\tn x_{k}}) \tn_{A^{\op}} (\omega_{l},y_{j})\sbt x_{l}\sbt (y, \omega_{k})
\\ \equiv& \ov{x_{k}} \tn_{A^{\op}} [y,\ov{b_{k}} ] - 1 \tn_{A^{\op}} \ov{x_{k}}\sbt [ y,\ov{b_{k}} ]
\end{align*}
As mentioned earlier $H\XA^{1}$ is in fact also a right Hopf algebroid, as described in Section \ref{SPrelimHopf}. We can describe its right Hopf algebroid in a similar manner: 
\begin{align*}
 x_{[+]}  \tn  x_{[-]} =&   y_{j}  \tn   (\rho_{j},x)- 1  \tn  y_{j}\sbt(\rho_{j},x)\equiv  y_{j}  \tn   [\ov{x},c_{j}]   - 1  \tn  y_{j}\sbt [\ov{x},c_{j}]
 \\ \ov{y}_{[+]}  \tn  \ov{y}_{[-]} \equiv&   (\omega_{i},y)_{[+]}\sbt(x_{i})_{[+]}  \tn  (x_{i})_{[-]}\sbt (\omega_{i},y)_{[-]} 
\\  =& (\omega_{k},y)\sbt y_{j}  \tn   (\rho_{j},x_{i})\sbt (x_{k},\omega_{i})- (\omega_{k},y)  \tn  y_{j}\sbt(\rho_{j},x_{i})\sbt (x_{k},\omega_{i})
\\  =& (\omega_{k},y)\sbt y_{j}  \tn   \evr(\rho_{j}\tn x_{k})- (\omega_{k},y)  \tn  y_{j}\sbt \evr(\rho_{j}\tn x_{k})
\\  =& (\omega_{k},y)\sbt x_{k}  \tn  1- (\omega_{k},y)  \tn  x_{k} \equiv  \ov{y}  \tn   1 - [\ov{y},b_{i}] \tn  x_{i}
\end{align*}
Note that for a DGA where $\wedge: \Omega^{1}\tn \Omega^{1}\rightarrow \Omega^{2}$ is surjective, $\ct{D}\XA$ is a quotient of $H\XA^{1}$ as an algebra and its Hopf algebroid structure, as described in \cite{ghobadi2020hopf}, is defined as the image of the above maps.

Lastly, we recall from Theorem 4.13 of \cite{ghobadi2020hopf}, that $H\XA^{1}$ admits an invertible antipode if there exists a linear map $\Upsilon :\XA^{1}\rightarrow A$ satisfying $\Upsilon (xa)=\Upsilon (x)a+\evl (x\otimes da)$ and $\Upsilon (ax)=a\Upsilon (x)+\evr (da \otimes x) $ for arbitrary $a\in A$ and $x\in \XA^{1}$. If such a morphism exists, we can simplify the action of the antipode by re-writing it in the case of a surjective calculus: 
\begin{align*}
S(x)&=-(\omega_{i},x)\sbt x_{i}-\ov{\Upsilon (x) }\equiv -\ov{x} -\ov{\Upsilon (x) }
\\S(\ov{y})&\equiv S ((db_{i},y)\sbt x_{i}) = S(x_{i}) \sbt  S ((db_{i},y)) = \left( -(\omega_{j},x_{i})\sbt x_{j}-\ov{\Upsilon (x_{i}) }\right)\sbt (y,db_{i})=  
\\&\equiv -\ov{x_{i}} \sbt [ y,\ov{b_{i}}] - \ov{\Upsilon (x_{i}) }\sbt [ y,\ov{b_{i}}]= \ov{b_{i}x_{i}} \sbt y -\ov{x_{i}} \sbt  y\sbt \ov{b_{i}} +\ov{\Upsilon (b_{i}x_{i}) }\sbt  y - \ov{\Upsilon (x_{i}) }\sbt  y\sbt \ov{b_{i}} 
\end{align*}
Similarly, the inverse of the anitpode is defined by $S^{-1} ( \ov{x}) = -x - \Upsilon (x) $ and  
\begin{align*}
S^{-1}(x)&=-\left( y_{j}+\Upsilon (y_{j})\right)\sbt(\rho_{j},x)\equiv -\left( y_{j}+\Upsilon (y_{j})\right)\sbt [\ov{x}, c_{j}]
\\&= (y_{j}c_{j})\sbt \ov{x} +\Upsilon (y_{j}c_{j})\sbt \ov{x} - y_{j} \sbt \ov{x} \sbt c_{j} - \Upsilon (y_{j}) \sbt \ov{x} \sbt c_{j}
\end{align*}
Finally, we conclude this section by recalling that the antipode of $H\XA^{1}$ extends to $\ct{D}\XA$ if $\Upsilon$ satisfies 
\begin{equation}\label{EqUpFla}
\Upsilon (\evl^{2}(x^{2}\otimes d\omega_{i})x_{i} ) +\Upsilon \big(\Upsilon (\evl^{2}(x^{2}\otimes \omega_{j}\wedge \omega_{k})x_{k} )x_{j}\big)=0
\end{equation}
for all $x^{2}\in \XA^{2} $. This result appears in Theorem 5.5 of \cite{ghobadi2020hopf} which has two components, equation \ref{EqUpFla} and equation (63) in \cite{ghobadi2020hopf}, the latter of which automatically holds when the calculus is surjective and \ref{EqUpFla} holds.  
\subsection{Isotopy Quotients of $H\XA^{1}$ and $\ct{D}\XA$}\label{SBialgIso}
In this section we study the isotopy quotients of $H\XA^{1}$ and $\ct{D}\XA$ and provide general formulae for the relations of these algebras. Before we look at general calculi, consider any \emph{classical} DGA, where $A$ is a commutative algebra and $\Omega^{n}$ are symmetric bimodules with their canonical pivotal structure $\evr ^{n}= \evl^{n}(\flip)$ and $\cvr^{n}=(\flip)\cvl^{n}$, for $n=1,2$. In the isotopy quotient of $H\XA^{1}$, we have that $\evl (x\tn da)=[x,a]= [x,\ov{a}]= (x,da)$ and by relation \eqref{EqRelInv1} that $(da,y)=\evl (y\tn da)$. Relations \eqref{EqRelInv2}, \eqref{EqRelHpf1} and \eqref{EqRelHpf2} hold automatically and we can conclude that $\iso{H\XA^{1}}\cong T\XA^{1}_{\sbt}$ as an algebra and $T\XA^{1}_{\sbt}$ has a natural Hopf algebroid structure in this case. More specifically, in the final section of \cite{ghobadi2020hopf} we considered the case of a Lie-Rinehart algebra $(A,\XA^{1}, \tau, [,])$, where $\XA^{1}$ is fgp as an $A$-module. From this data we obtain a DGA with a cannonical pivotal structure such that the quotient of $\ct{D}\XA$ by the ideal $\langle s(a) -t(a)\rangle_{a \,\in\, A}$ recovers the \emph{universal enveloping Hopf algebroid} of $(A,\XA^{1}, \tau, [,])$, denoted by $V(A,\XA^{1})$, as an algebra. We remarked that the Hopf algebroid structure of $V(A,\XA^{1})$, which can be found in \cite{kowalzig2009hopf}, is clearly a projection of the structure maps on $\ct{D}\XA$. The formalism of the isotopy quotient which we presented in Section \ref{SIso} makes this statement more concrete and shows that we recover the Hopf algebroid $(A,V(A,\XA^{1}))$ as the isotropy quotient of $(A,\ct{D}\XA)$. At the end of this section, we will show this statement directly. 

Now we consider an arbitrary surjective calculus with a pivotal structure as in Section \ref{SSurjAlg} and look at $\iso{H\XA^{1}}$. Since the source and target maps are equal, meaning $a=\ov{a}$, in $\iso{H\XA^{1}}$ we will no longer use the over line notation. We will also abuse notation and write $a$ for the image of element $a\in A$ in $\ab{A}$. 

\begin{prop}\label{PIsoHX} As an algebra, $\iso{H\XA^{1}}$ will be the quotient of $\ab{A}\star T_{\K}\XA^{1} \star T_{\K}\fr{Y}^{1}$ by the relations \eqref{EqAX}, \eqref{EqXA}, \eqref{EqTXop} and 
\begin{align}
\sum_{m}&s_{m} xr_{m}= 0 = \sum_{m}\ov{s_{m} xr_{m}} \label{EqIsoBim}
\\ &\ov{y} \sbt (b_{i}x_{i})-\ov{yb_{i}}\sbt x_{i} = \ov{y}\label{EqIsoInv1}
\\ &  x\sbt \ov{y_{j}c_{j}} - (c_{j}x)\sbt \ov{y_{j}} = x\label{EqIsoInv2} 
\\ y_{j}\sbt \big( \ov{xc_{j}a} -\ov{c_{j}xa}\big)& - (ay_{j})\sbt \big( \ov{xc_{j}} -\ov{c_{j}x}\big) =ax-xa  \label{EqIsoHpf1}
\\ \ov{x_{i}a}\sbt \big(xb_{i}-b_{i}x\big)&- \ov{x_{i}}\sbt \big(axb_{i}-ab_{i}x\big)= \ov{ax}-\ov{xa}\label{EqIsoHpf2}
\end{align}
where $\sum_{m} r_{m}\tn_{\K} s_{m}\in \ct{N}$
\end{prop}
\begin{proof} We simply need to review the relations describing $H\XA^{1}$ with the additional relation that $a=\ov{a}$ and $a\sbt b = b\sbt a$ for $a,b\in A$. Relation \eqref{EqIsoBim} follow directly from relations \eqref{EqSurjBXA} and \eqref{EqSurjInvDef} with the additional observation that 
\begin{equation*}
\sum_{m} s_{m}x \sbt r_{m}= \sum_{m}\big[ (s_{m} xr_{m}) + \evl (s_{m} x\tn dr_{m})\big] + \sum_{m}\big[ (s_{m} xr_{m}) + \evl ( x\tn dr_{m})s_{m}\big]=\sum_{m} (s_{m} xr_{m})
\end{equation*}
and a symmetric equality for $\sum_{m}\ov{s_{m} xr_{m}}= \sum_{m}\ov{ xr_{m}}\sbt s_{m}$ hold. Relations \eqref{EqIsoInv1} and \eqref{EqIsoInv2} follow directly from \eqref{EqSurjInv1} and \eqref{EqSurjInv2}, respectively, and the assumption that $\ov{a}\sbt x=a\sbt x= ax$. For \eqref{EqIsoHpf1} and \eqref{EqIsoHpf2}, we apply the following arguments to \eqref{EqSurjHpf1} and \eqref{EqSurjHpf2}:
\begin{align*}
[y_{j},a]\sbt [\ov{x},c_{j}]  =&\big((y_{j}a)+ \evl (y_{j}\tn da )-(ay_{j})\big) \sbt \big( \ov{c_{j}x}+ \evr (dc_{j}\tn x )-\ov{xc_{j}}  \big)
\\ =&\big((y_{j}a)+ \evl (y_{j}\tn da )-(ay_{j})\big) \sbt \big( \ov{c_{j}x}-\ov{xc_{j}} \big)  +\big( y_{j}\sbt a- (ay_{j})\big)\sbt \evr (dc_{j}\tn x )
\\=& \big((y_{j}a)+ \evl (y_{j}\tn da )-(ay_{j})\big) \sbt \big( \ov{c_{j}x}-\ov{xc_{j}} \big) +  x\sbt a + \evl ( y_{j} \tn d\evr (dc_{j}\tn x ))\sbt a 
\\& -ax - a\sbt \evl ( y_{j} \tn d\evr (dc_{j}\tn x ))
\\=& \big((y_{j}a)+ \evl (y_{j}\tn da )-(ay_{j})\big) \sbt \big( \ov{c_{j}x}-\ov{xc_{j}} \big)+xa - ax  + \evl(x\tn da) 
\end{align*}
\begin{align*}
[\ov{x_{i}},a]\sbt [ x, b_{i}]=& \big(\ov{ax_{i}}+\evr (da \tn x_{i} )-\ov{x_{i}a}\big) \sbt \big((xb_{i})+ \evl (x\tn db_{i} )-(b_{i}x)\big)
\\=& \big(\ov{ax_{i}}+\evr (da \tn x_{i} )-\ov{x_{i}a}\big) \sbt \big((xb_{i})-(b_{i}x)\big) + \big(\ov{x_{i}}\sbt a -\ov{x_{i}a}\big)\sbt \evl (x\tn db_{i} )
\\=& \big(\ov{ax_{i}}+\evr (da \tn x_{i} )-\ov{x_{i}a}\big) \sbt \big((xb_{i})-(b_{i}x)\big)+ \ov{x} \sbt a +  \evr (d\evl (x\tn db_{i} )\tn x_{i} )\sbt a  
\\ & -\ov{xa} - \evr (d\evl (x\tn db_{i} )\tn x_{i} a) 
\\=& \big(\ov{x_{i}a}-\evr (da \tn x_{i} )-\ov{ax_{i}}\big) \sbt \big((xb_{i})-(b_{i}x)\big)+\ov{ax}-\ov{xa} + \evr ( da\tn x)\qedhere
\end{align*}
\end{proof}
Following from Corollary \ref{CDXA} we can describe $\iso{\ct{D} \XA}$ as a quotient of $\iso{H\XA^{1}} $:
\begin{corollary}\label{CIsoDXA} If $(A, d,\wedge , \Omega^{n})$ is a surjective calculus and pivotal, then $\iso{\ct{D} \XA}$, is the quotient of $\iso{H\XA^{1}} $ by relations \eqref{EqSujExt1}, \eqref{EqSujExt2} and \eqref{EqFlat}. 
\end{corollary}
Unfortunately at the moment of writing this work, the author is not aware of a simplification of relations \eqref{EqSujExt1}, \eqref{EqSujExt2} in $\iso{\ct{D} \XA}$. However as mentioned in Section \ref{SSurjective}, when the DGA is the maximal prolongation of $(A,\Omega^{1},d)$, these relations become redundant, and the examples which concern us in this work are exactly of this form. 

Now we us review the coring structure on $\iso{\ct{D}\XA}$ and $\iso{H\XA^{1}}$. Since the source and target maps are now equal in this setting, the comultiplication is a map from the left $\ab{A}$-module to $ \iso{H\XA^{1}}$ to the tensor product of this left $\ab{A}$-module with itself:
\begin{align*}
\Delta (x) = & x\otimes 1 +  x \sbt \ov{b_{i}} \otimes x_{i}= x\otimes 1 +[ x , b_{i}] \otimes x_{i}= x\otimes 1 - b_{i}x \otimes x_{i} +xb_{i} \otimes x_{i}+ \evl(x\tn db_{i}) \otimes x_{i}
\\&= x\otimes 1 - x \otimes b_{i}x_{i} +xb_{i} \otimes x_{i}+ 1 \otimes \evl(x\tn db_{i})\sbt x_{i}= x\otimes 1 - x \otimes b_{i}x_{i} +xb_{i} \otimes x_{i}+ 1 \otimes x
\end{align*}
In a similar manner we also have that $\Delta (\ov{x})= 1\tn \ov{y}+  \ov{y_{j}c_{j}}\tn  \ov{y}+\ov{y_{j}}\tn  \ov{c_{j}y} - \ov{y}\tn 1$ and the Hopf algebroid structure is defined by
\begin{align*}
x_{(+)} \tn_{\ab{A}} x_{(-)} =&x \tn_{\ab{A}} 1 - [x,c_{j}] \tn_{\ab{A}} \ov{y_{j}} =x \tn_{\ab{A}} 1 - 1\tn_{\ab{A}} x + c_{j}x \tn_{\ab{A}} \ov{y_{j}} - xc_{j} \tn_{\ab{A}} \ov{y_{j}}  
\end{align*}
Finally observe that for a Lie-Rinehart pair $(A,\XA^{1}, \tau, [,])$ with fgp $\XA^{1}$, relations \eqref{EqIsoBim}, \eqref{EqIsoHpf1} and \eqref{EqIsoHpf2} in Proposition \ref{PIsoHX} all become trivial since $\XA^{1}$ is a symmetric bimodule. In fact, relation \eqref{EqIsoInv1} simplifies as
$$\ov{y} \sbt (b_{i}x_{i})-\ov{yb_{i}}\sbt x_{i} = \ov{y} \sbt b_{i} \sbt x_{i}-\ov{b_{i}y}\sbt x_{i} = \evr (db_{i}\tn y) \sbt x_{i} = y $$
and by a symmetric argument both \eqref{EqIsoHpf1} and \eqref{EqIsoHpf2} reduce to $x=\ov{x}$. Since $\evr( da\tn x) = \evl (x\tn da)$ in this setting, relations in \eqref{EqTXop} become redundant and $\iso{H\XA^{1}}$ simply becomes isomorphic to $T\XA^{1}_{\sbt}$. Recall from \cite{ghobadi2020hopf}, that the $\XA^{2}= \bigwedge^{2}(\Omega^{1})$ and is the maximal prolongation of $(A,\Omega^{1},d)$ in this case. Hence relations \eqref{EqSujExt1} and \eqref{EqSujExt1} become redundant. Hence, $\iso{\ct{D}\XA}$ is simply the quotient of $T\XA^{1}_{\sbt}$ by relation \eqref{EqFlat}. As seen in Section 5.3 of \cite{ghobadi2020hopf}, this relation reduces to $x\sbt y -y\sbt x =[x,y]$ where $[x,y]$ denotes the Lie bracket on $\XA^{1}$ and we recover $V(A,\XA^{1})$ as an algebra. As a Hopf algebroid, the previous formulae for $\iso{\ct{D}\XA}$ simplify to the desired structure of $V(A,\XA^{1})$ since $\XA^{1}$ is a symmetric bimodule. 
\section{Calculus for Digraphs}\label{SGrph}
Let $(V,E)$ be a pair of finite sets defining a directed graph i.e. there exist maps $s,t:E \rightarrow V$ sending each $e$ element of $E$, called an \emph{edge}, to its source vertex $s(e)$ and its target vertex $t(e)$, respectively. Throughout this work, unless mentioned otherwise, we will assume both that no edge has the same source and target, and that there is at most one edge between any vertices. Hence, any edge is uniquely determined by its source and target and $e$ can alternatively be denote by $s(e)\rightarrow t(e)$. Any such directed graph defines a surjective first order calculus and in this section we will describe $H\XA^{1}$ for this calculus.

\textbf{Notation.} To ease computation, we will use the notation $a\nrightarrow b$ instead of $\arr{a}{b}\notin E$ and use 
\begin{align}
\delta_{a\rightarrow b} = \begin{cases}1 \ \text{if}\ a\rightarrow b\in E
\\0 \ \text{otherwise}
\end{cases},
\quad 
\delta_{a\nrightarrow b} = \begin{cases}1 \ \text{if}\ a\rightarrow b\notin E
\\0 \ \text{otherwise}
\end{cases}
\end{align}
throughout our computation. 

Let $A= \K (V)=\lbrace f:V\rightarrow \K \rbrace$ be the algebra of functions on $V$. It follows that $A$ is a semisimple algebra with a finite basis $\lbrace f_p \mid p\in V\rbrace$, where $f_{p}(q)=\delta_{p,q}$ for any $p,q\in V$, which form a complete set of idempotents $\lbrace f_p \mid p\in V\rbrace$ for $A$. 

The bimodule $\Omega^{1}:= \K. E$ has basis elements $\omega_{\arr{a}{b}}$ corresponding to edges $\arr{a}{b}\in E$, with its bimodule actions defines by $f_{p}. \omega_{\arr{a}{b}} . f_{q}= \delta_{p,a}\delta_{b,q}\omega_{\arr{a}{b}}$. The differential of the calculus is defined by  
\begin{equation}
df = \sum_{\arr{a}{b}\in E} [f(b)-f(a)].\omega_{\arr{a}{b}}\ \Longleftrightarrow\ df_{p} = \sum_{a:\arr{a}{p}\in E} \omega_{\arr{a}{p}} - \sum_{b:\arr{p}{b}\in E} \omega_{\arr{p}{b}}
\end{equation}
The dual bimodule $\XA^{1}$ is spanned, as a vectorspace, by elements of the form $e_{b\leftarrow a}$ with its $A$-bimodule structure defined by $f_{p}. e_{b\leftarrow a}. f_{q}= \delta_{p,b}\delta_{a,q}e_{b\leftarrow a}$. The calculus carries a canonical pivotal structure defined by
\begin{align}
\cvl(1)&= \sum _{\arr{s}{t}\in E} \omega_{\arr{s}{t}}\otimes e_{\arl{s}{t}},\hspace{1cm} \evl(e_{\arl{a}{b}}\otimes \omega_{\arr{p}{q}}) = \delta_{p,a}.\delta_{q,b} f_{b}
\\\cvr(1)&= \sum _{\arr{s}{t}\in E} e_{\arl{s}{t}}\otimes \omega_{\arr{s}{t}},\hspace{1cm} \evr(\omega_{\arr{p}{q}}\otimes e_{\arl{a}{b}}) = \delta_{p,a}.\delta_{q,b} f_{a}
\end{align}
For arbitrary $\arr{a}{b},\arr{p}{q}\in E$. 

Notice that the calculus is surjective and that in the notation of Section \ref{SSurjAlg}, $\ker (.) =\K.\{ f_{p}\tn_{\K}f_{q} \mid p,q \in V, \ p\neq q\} $ and $\ct{N}= \K.\{ f_{p}\tn_{\K}f_{q} \mid  \ p\neq q,\ \arr{p}{q}\notin E \} $ so that $\Omega^{1}= \ker (.)/\ct{N}$. Additionally, by the definition of $d$ we have that $\omega_{\arr{a}{b}}= f_{a}.(df_{b})= -(df_{a}).f_{b}$ and thereby
$$\cvl (1)= db_{i}\tn x_{i}= \sum _{\arr{a}{b}\in E} -df_{a} \otimes e_{\arl{a}{b}} , \hspace{1cm}\cvr(1)= y_{i}\tn dc_{i} =\sum _{\arr{a}{b}\in E} e_{\arl{a}{b}}\otimes df_{b} $$ 
With this information, we can describe $H\XA^{1}$ in terms of generators and relations by Corollary \ref{CSurjHXA}. First let us write out relations \eqref{EqAX}, \eqref{EqXA}, \eqref{EqBXA}, \eqref{EqSurjBXA}, \eqref{EqSurjInvDef} for the algebra $H\XA^{1}$. The algebra will be generated by elements $f_{p}\in A$, $\ov{f_{q}}\in A^{\op}$, $e_{\arl{a}{b}}$, $\ov{e_{\arl{p}{q}}}$ corresponding to edges $\arr{a}{b}, \arr{p}{q}\in E$ and the mentioned relations translate as follows: 
\begin{align}
f_{p}\sbt f_{q}= \delta_{p,q}f_{p},\quad  \ov{f_{p}}\sbt\ov{f_{q}}&=\delta_{p,q}f_{p}, \quad f_{p}\sbt \ov{f_{q}}= \ov{f_{q}}\sbt f_{p}
\\  f_{p}\sbt  e_{\arl{a}{b}}= \delta_{p,b}e_{\arl{a}{b}},& \quad  \ov{f_{q}}\sbt\ov{e_{\arl{a}{b}}}= \delta_{a,q}\ov{e_{\arl{a}{b}}}
\\  e_{\arl{a}{b}}\sbt  f_{p}= \delta_{p,a}(e_{\arl{a}{b}}-f_{b})+\delta_{b,p}f_{b},& \quad  \ov{e_{\arl{a}{b}}}\sbt \ov{f_{q}}= \delta_{q,b}(\ov{e_{\arl{a}{b}}} + \ov{f_{a}}) -\delta_{q,a} \ov{f_{a}}
\\ f_{p}\sbt \ov{e_{\arl{a}{b}}}\sbt  f_{q} = &0 = \ov{f_{q}}\sbt e_{\arl{a}{b}}\sbt \ov{f_{p}}, \ \text{ when } p\nrightarrow q \text{ and }p\neq q
\end{align}
Since $1= \sum_{p,q\in V} f_{p}\sbt \ov{f_{q}}$, we can observe that
\begin{align*}
e_{\arl{s}{t}}&=\left(\sum_{p,q\in V} f_{p}\sbt \ov{f_{q}}\right) \sbt e_{\arl{s}{t}} \sbt \left(\sum_{a,b\in V} f_{a}\sbt \ov{f_{b}}\right)= \sum_{q,b\in V}  \ov{f_{q}}\sbt e_{\arl{s}{t}} \sbt \ov{f_{b}}+\sum_{b,q,a\in V}  \delta_{b,q}.\ov{f_{q}}\sbt (\delta_{a,t}f_{t}-\delta_{a,s}f_{t}) 
\\&= \sum_{\arr{b}{q}\in E}  \ov{f_{q}}\sbt e_{\arl{s}{t}} \sbt \ov{f_{b}}+ \sum_{q\in V}  \ov{f_{q}}\sbt e_{\arl{s}{t}} \sbt \ov{f_{q}}
\end{align*}
and similarly $\ov{e_{\arl{s}{t}}}= \sum_{\arr{p}{q}\in E}  f_{p}\sbt \ov{e_{\arl{s}{t}}} \sbt f_{q}+ \sum_{q\in V}  f_{q}\sbt \ov{e_{\arl{s}{t}}} \sbt f_{q}$. Hence, we can introduce a new set of generators for the algebra: 
\begin{align}
&\points{p}{q} :=f_{p}\sbt \ov{f_{q}},\quad\quad (\genleft{a}{b}{c}{d}):= - \ov{f_{d}}\sbt e_{\arl{a}{b}} \sbt \ov{f_{c}}, \quad\quad  (\genright{a}{b}{c}{d}):= f_{a}\sbt \ov{e_{\arr{c}{d}}} \sbt f_{b}
\\& (\genlefty{a}{b}{q}):=  \ov{f_{q}}\sbt e_{\arl{a}{b}} \sbt \ov{f_{q}}- f_{b}\sbt \ov{f_{q}}, \quad \quad (\genrighty{c}{d}{p}):=f_{p}\sbt \ov{e_{\arr{c}{d}}} \sbt f_{p} +  f_{p}\sbt \ov{f_{c}}
\end{align}
corresponding to vertices $p,q\in V$ and edges $\arr{a}{b},\arr{c}{d}\in E$, so that the above relations are simplified to 
\begin{align*}
 \points{p}{q}\sbt (\genlefty{a}{b}{c})\sbt \points{e}{f}=\delta_{p,b}.\delta_{q,c,f}.\delta_{e,a}. (\genlefty{a}{b}{c}),& \quad \points{p}{q}\sbt (\genrighty{c}{d}{a})\sbt \points{e}{f}=\delta_{p,a,e}.\delta_{q,c}.\delta_{f,d}. (\genrighty{c}{d}{a})
\\\points{p}{q}\sbt (\genleft{a}{b}{c}{d})\sbt \points{e}{f}=\delta_{p,b}.\delta_{q,d}.\delta_{e,a}.\delta_{f,c}. (\genleft{a}{b}{c}{d}),& \quad \points{p}{q}\sbt (\genright{a}{b}{c}{d})\sbt \points{e}{f}=\delta_{p,a}.\delta_{q,c}.\delta_{e,b}.\delta_{f,d}. (\genright{a}{b}{c}{d})
\end{align*}
We leave the verifications of these simple relations to the reader. From this new set of generators, it should be clear that $H\XA^{1}$ can be viewed as a quotient of the path algebra of the double of the digraph, $\K D^{\mathrm{db}}$. Additionally, the previous generators can be recovered as follows:
\begin{align}
e_{\arl{s}{t}}=&\sum_{q\in V}(\genlefty{s}{t}{q})-\sum_{\arr{a}{b}\in E}(\genleft{s}{t}{a}{b}) +\sum_{q\in V}\points{t}{q} 
 \\\ov{e_{\arl{s}{t}}}=&\sum_{q\in V}(\genrighty{s}{t}{q})+ \sum_{\arr{a}{b}\in E}(\genright{a}{b}{s}{t}) -\sum_{q\in V}\points{q}{s} 
\end{align}
\begin{thm}\label{THDig} The Hopf algebroid $H\XA^{1}$ arising from the differential calculus of a diagraph $D=(E,V)$ is isomorphic to the quotient of the path algebra $\K D^{\mathrm{db}}$ by the following additional relations  
\begin{align}
\sum_{i: \arr{q}{i} \in E}(\genright{a}{b}{q}{i})\sbt (\genleft{p}{b}{q}{i}) = \delta_{a,p}&. \points{a}{q}= \sum_{i: \arr{q}{i} \in E}(\genright{q}{i}{a}{b})\sbt (\genleft{q}{i}{p}{b}) \label{EqDigH1}
\\ \sum_{i:\arr{i}{q} \in E}(\genleft{d}{a}{i}{q})\sbt (\genright{d}{p}{i}{q}) = \delta_{a,p}&.\points{a}{q}= \sum_{i:\arr{i}{q} \in E}(\genleft{i}{q}{d}{a})\sbt (\genright{i}{q}{d}{p})\label{EqDigH2}
\\(\genrighty{s}{t}{q})= \sum_{i:\arr{q}{i}\in E}(\genright{q}{i}{s}{t})\sbt (\genlefty{q}{i}{t}),&\ \ \sum_{i:\arr{i}{q}\in E} (\genleft{s}{t}{i}{q})\sbt (\genrighty{i}{q}{s})= (\genlefty{s}{t}{q})\label{EqDigH3}
\end{align}
for arbitrary $q\in V$, $\arr{s}{t}\in E$ and pairs $ \arr{a}{b},\arr{p}{b}\in E $ and $ \arr{d}{a}, \arr{d}{p} \in E$. 
\end{thm}
\begin{proof} All we must do is write out relations \eqref{EqSurjInv1}, \eqref{EqSurjInv2}, \eqref{EqSurjHpf1} and \eqref{EqSurjHpf2} in terms of the new generators. We first observe that \eqref{EqSurjInv1} simplifies as follows for $\ov{e_{\arl{s}{t}}}\in \mathfrak{Y}^{1}$:
\begin{align*}
  \ov{e_{\arl{s}{t}}}=& \sum_{\arr{a}{b}\in E}[\ov{e_{\arl{s}{t}}},-f_{a}]\sbt e_{\arl{a}{b}}=\sum_{\arr{a}{b}\in E}f_{a}\sbt \ov{e_{\arl{s}{t}}}\sbt e_{\arl{a}{b}}-\sum_{\arr{a}{b}\in E} \ov{e_{\arl{s}{t}}}\sbt (f_{a}\sbt e_{\arl{a}{b}})
  \\=&\sum_{\arr{a}{b}\in E}f_{a}\sbt \ov{e_{\arl{s}{t}}}\sbt e_{\arl{a}{b}}
\end{align*}
We apply our change of variables to this equation:
\begin{align*}
 \sum_{\arr{a}{b}\in E}(\genright{a}{b}{s}{t})&+\sum_{q\in V}(\genrighty{s}{t}{q}) -\sum_{q\in V}\points{q}{s} 
\\=&\sum_{\substack{\arr{a}{b}\in E\\r\in V}}\points{a}{r}\sbt \left(\sum_{\arr{c}{d}\in E}(\genright{c}{d}{s}{t})+\sum_{q\in V}(\genrighty{s}{t}{q}) -\sum_{q\in V}\points{q}{s}  \right)\sbt \left(\sum_{p\in V}(\genlefty{a}{b}{p})-\sum_{\arr{e}{f}\in E}(\genleft{a}{b}{e}{f}) +\sum_{p\in V}\points{b}{p} \right)
\\=&\sum_{\arr{a}{b}\in E}\left(\genright{a}{b}{s}{t})\sbt (\genlefty{a}{b}{t})-\sum_{e:\arr{e}{t}\in E}(\genright{a}{b}{s}{t}) \sbt (\genleft{a}{b}{e}{t})+ (\genright{a}{b}{s}{t})\right) 
\end{align*}
We can now use the path algebra structure for the new generators. We first multiply the equation by $\points{q}{s}$ on the left and $\points{q}{e}$ on the right, for fixed $q\in V$ and $e\in V$ such that $\arr{e}{t}\in E$, to obtain 
\begin{equation}\label{EqInProof1}
\sum_{i:\arr{q}{i}\in E}(\genright{q}{i}{s}{t}) \sbt (\genleft{q}{i}{e}{t})= \delta_{s,e}\points{q}{s} 
\end{equation}
and secondly, we multiply the equation by $\points{q}{s}$ on the left and $\points{q}{t}$ on the right to obtain: 
\begin{equation}\label{EqInProof2}
(\genrighty{s}{t}{q})= \sum_{i:\arr{q}{i}\in E}(\genright{q}{i}{s}{t})\sbt (\genlefty{q}{i}{t})
\end{equation}
Hence, \eqref{EqSurjInv1} is equivalent to \eqref{EqInProof1} and \eqref{EqInProof2} holding in terms of the new generators. Now we look at \eqref{EqSurjInv2} for $e_{\arl{s}{t}}\in \XA^{1} $ and simplify it as $e_{\arl{s}{t}}=\sum_{\arr{a}{b}\in E} [e_{\arl{s}{t}} ,  \ov{f_{b}}]\sbt \ov{e_{\arl{a}{b}}}= -\sum_{\arr{a}{b}\in E} \ov{f_{b}}\sbt e_{\arl{s}{t}}\sbt \ov{e_{\arl{a}{b}}} $. We then apply our change of generators to this equation:
\begin{align*}
 \sum_{q\in V}(\genlefty{s}{t}{q})&-\sum_{\arr{c}{d}\in E}(\genleft{s}{t}{c}{d}) +\sum_{q\in V}\points{t}{q}=  
\\ & -\sum_{\substack{\arr{a}{b}\in E,\\r\in V}} \points{r}{b}\sbt \left(\sum_{q\in V}(\genlefty{s}{t}{q}) -\sum_{\arr{c}{d}\in E}(\genleft{s}{t}{c}{d}) +\sum_{q\in V}\points{t}{q}\right)\sbt \left( \sum_{\arr{e}{f}\in E}(\genright{e}{f}{a}{b})+ \sum_{p\in V}(\genrighty{a}{b}{p}) -\sum_{p\in V}\points{p}{a} \right) 
\\=&\sum_{\arr{a}{b}\in E}  \left( \sum_{f:\arr{s}{f}\in E}(\genleft{s}{t}{a}{b}) \sbt (\genright{s}{f}{a}{b})+(\genleft{s}{t}{a}{b})\sbt (\genrighty{a}{b}{s}) -(\genleft{s}{t}{a}{b})\right)
\end{align*}
Similar to \eqref{EqSurjInv1}, we use the path algebra structure and see that the equation is equivalent to the following relations holding: 
\begin{equation*}
\sum_{i:\arr{i}{q}\in E}(\genleft{s}{t}{i}{q}) \sbt (\genright{s}{f}{i}{q})= \delta_{f,t}\points{t}{q},\quad \sum_{i:\arr{i}{q}\in E} (\genleft{s}{t}{i}{q})\sbt (\genrighty{i}{q}{s})= (\genlefty{s}{t}{q})
\end{equation*}
Now we expand relation \eqref{EqSurjHpf1} for fixed pair $e_{\arl{s}{t}}\in \XA^{1}$ and $q\in V$:
\begin{align*}
\delta_{q,t} \ov{f_{t}}-\delta_{q,s} \ov{f_{t}}=&\ov{\evl (e_{\arl{s}{t}}\tn df_{q})}= \sum_{\arr{a}{b}\in E} [e_{\arl{a}{b}},\ov{f_{q}}] \sbt [\ov{e_{\arl{s}{t}}},f_{b}]  
\\=&\sum_{\arr{a}{b}\in E} \delta_{q,s}e_{\arl{a}{b}} \sbt \ov{e_{\arl{s}{t}}}\sbt f_{b}-\sum_{\arr{a}{b}\in E} \ov{f_{q}}\sbt e_{\arl{a}{b}} \sbt \ov{e_{\arl{s}{t}}}\sbt f_{b}- \sum_{\arr{a}{b}\in E} [e_{\arl{a}{b}}\sbt f_{b},\ov{f_{q}}] \sbt \ov{e_{\arl{s}{t}}} 
\\=&\sum_{\arr{a}{b}\in E} \delta_{q,s}e_{\arl{a}{b}} \sbt \ov{e_{\arl{s}{t}}}\sbt f_{b}-\sum_{\arr{a}{b}\in E} \ov{f_{q}}\sbt e_{\arl{a}{b}} \sbt \ov{e_{\arl{s}{t}}}\sbt f_{b}
\end{align*}
We use the following observation for arbitrary $\arr{a}{b}\in E$ to ease calculations:
\begin{align*}
e_{\arl{a}{b}} \sbt \ov{e_{\arl{s}{t}}}\sbt f_{b}&= \left( \sum_{p\in V}(\genlefty{a}{b}{p})-\sum_{\arr{c}{d}\in E}(\genleft{a}{b}{c}{d}) +\sum_{p\in V}\points{b}{p}\right) \sbt \left(\sum_{e:\arr{e}{b}\in E}(\genright{e}{b}{s}{t})+ (\genrighty{s}{t}{b}) -\points{b}{s}\right)
\\&= -\sum_{d:\arr{s}{d}\in E}(\genleft{a}{b}{s}{d})\sbt  (\genright{a}{b}{s}{t})+ (\genlefty{a}{b}{s})\sbt (\genright{a}{b}{s}{t})+ (\genrighty{s}{t}{b}) -\points{b}{s}
\end{align*}
Thereby we re-write equation \eqref{EqSurjHpf1} in terms of our new generators:
\begin{align*}
\delta_{q,t} \sum_{p\in V}\points{p}{t}-\delta_{q,s} \sum_{p\in V}\points{p}{t}=&\delta_{q,s}\sum_{\arr{a}{b}\in E}\left(\sum_{d:\arr{s}{d}\in E}-(\genleft{a}{b}{s}{d})\sbt  (\genright{a}{b}{s}{t})+ (\genlefty{a}{b}{s})\sbt (\genright{a}{b}{s}{t})+ (\genrighty{s}{t}{b}) -\points{b}{s}\right)
\\&+\delta_{\arr{s}{q}\in E}.\sum_{\arr{a}{b}\in E}(\genleft{a}{b}{s}{q})\sbt  (\genright{a}{b}{s}{t}) - \delta_{q,s}\left(\sum_{\arr{a}{b}\in E}(\genlefty{a}{b}{s})\sbt (\genright{a}{b}{s}{t})+ (\genrighty{s}{t}{b}) -\points{b}{s}\right)
\\=& \delta_{\arr{s}{q}\in E}.\sum_{\arr{a}{b}\in E}(\genleft{a}{b}{s}{q})\sbt  (\genright{a}{b}{s}{t})-\delta_{q,s}\sum_{\arr{a}{b}\in E}\sum_{d:\arr{s}{d}\in E}(\genleft{a}{b}{s}{d})\sbt  (\genright{a}{b}{s}{t})
\end{align*}
As with the previous relations, we observe that this relation is equivalent to 
\begin{equation}
\sum_{i:\arr{i}{a}\in E}(\genleft{i}{a}{s}{d})\sbt  (\genright{i}{a}{s}{t})= \delta_{t,d}\points{a}{t}
\end{equation}
holding for arbitrary $a\in V$ and $\arr{s}{t}, \arr{s}{d}\in E$. A similar calculation follows for \eqref{EqSurjHpf2}:
\begin{align*}
\delta_{q,t}f_{s}-\delta_{q,s}f_{s}=&\evr (df_{q}\tn e_{\arl{s}{t}})= \sum_{\arr{a}{b}\in E} [\ov{e_{\arl{a}{b}}},f_{q}]\sbt [ e_{\arl{s}{t}}, -\ov{f_{a}}]
\\=& \sum_{\arr{a}{b}\in E}f_{q}\sbt \ov{e_{\arl{a}{b}}}\sbt e_{\arl{s}{t}}\sbt\ov{f_{a}} - \delta_{q,t}\sum_{\arr{a}{b}\in E} \ov{e_{\arl{a}{b}}}\sbt e_{\arl{s}{t}}\sbt\ov{f_{a}}+ \sum_{\arr{a}{b}\in E} [\ov{e_{\arl{a}{b}}}\sbt \ov{f_{a}} ,f_{q}]\sbt  e_{\arl{s}{t}}
\\=& \sum_{\arr{a}{b}\in E}f_{q}\sbt \ov{e_{\arl{a}{b}}}\sbt e_{\arl{s}{t}}\sbt\ov{f_{a}} - \delta_{q,t}\sum_{\arr{a}{b}\in E} \ov{e_{\arl{a}{b}}}\sbt e_{\arl{s}{t}}\sbt\ov{f_{a}}
\end{align*}
Next we observe the following identity for arbitrary $\arr{a}{b}\in E$:
\begin{align*}
\ov{e_{\arl{a}{b}}}\sbt e_{\arl{s}{t}}\sbt\ov{f_{a}}= & \left(  \sum_{\arr{e}{f}\in E}(\genright{e}{f}{a}{b})+ \sum_{r\in V}(\genrighty{a}{b}{r}) -\sum_{r\in V}\points{r}{a}\right)\sbt \left( -\sum_{d:\arr{a}{d}\in E}(\genleft{s}{t}{a}{d})+(\genlefty{s}{t}{a}) +\points{t}{a}\right) 
\\=&- \sum_{e:\arr{e}{t}\in E}(\genright{e}{t}{a}{b})\sbt (\genleft{s}{t}{a}{b}) - (\genrighty{a}{b}{t})\sbt(\genleft{s}{t}{a}{b})-(\genlefty{s}{t}{a})- \points{t}{a}
\end{align*}
Using this expansion, we can change variables:
\begin{align*}
\delta_{q,t}f_{s}-\delta_{q,s}f_{s}=&-\delta_{\arr{q}{t}\in E}\sum_{\arr{a}{b}\in E}(\genright{q}{t}{a}{b})\sbt (\genleft{s}{t}{a}{b})-\delta_{q,t}\sum_{\arr{a}{b}\in E}\left( (\genrighty{a}{b}{t})\sbt(\genleft{s}{t}{a}{b})+(\genlefty{s}{t}{a})+ \points{t}{a} \right)  
\\&+ \delta_{q,t}\sum_{\arr{a}{b}\in E} \left( \sum_{e:\arr{e}{t}\in E}(\genright{e}{t}{a}{b})\sbt (\genleft{s}{t}{a}{b}) + (\genrighty{a}{b}{t})\sbt(\genleft{s}{t}{a}{b})+(\genlefty{s}{t}{a})+ \points{t}{a}\right) 
\\ =&\delta_{q,t}\sum_{\arr{a}{b}\in E} \left( \sum_{e:\arr{e}{t}\in E}(\genright{e}{t}{a}{b})\sbt (\genleft{s}{t}{a}{b})\right)-\delta_{\arr{q}{t}\in E}\sum_{\arr{a}{b}\in E}(\genright{q}{t}{a}{b})\sbt (\genleft{s}{t}{a}{b})
\end{align*}
Thereby, \eqref{EqSurjHpf2} is equivalent to the following relations holding for arbitrary $a\in V$ and $\arr{s}{t}, \arr{e}{t}\in E$:
\begin{equation*}
 \sum_{i:\arr{a}{i}\in E}(\genright{e}{t}{a}{i})\sbt (\genleft{s}{t}{a}{i})= \delta_{e,s}\points{s}{a}\qedhere
\end{equation*}
\end{proof}
We can now write out the Hopf algebroid structure of $H\XA^{1}$ in terms of the new generators. First let us describe the coalgebra and counit action on $e_{\arl{s}{t}}$ and $\ov{e_{\arl{s}{t}}}$ by Section \ref{SBialSurj}: 
\begin{align*}
\Delta (e_{\arl{s}{t}}) = & e_{\arl{s}{t}}\otimes 1 -  \sum_{\arr{a}{b}\in E}[e_{\arl{s}{t}},\ov{f_{a}}]\otimes e_{\arl{a}{b}}=e_{\arl{s}{t}}\otimes 1 -  \sum_{\arr{a}{b}\in E} e_{\arl{s}{t}}\sbt \ov{f_{a}}\otimes f_{b}\sbt e_{\arl{a}{b}}
\\&+  \sum_{\arr{a}{b}\in E} \ov{f_{a}}\sbt e_{\arl{s}{t}}\otimes f_{b}\sbt e_{\arl{a}{b}} = e_{\arl{s}{t}}\otimes 1 -  \sum_{\arr{a}{b}\in E}\ov{f_{b}}\sbt e_{\arl{s}{t}}\sbt \ov{f_{a}}\otimes  e_{\arl{a}{b}}
\\ \Delta (\ov{e_{\arl{s}{t}}})=& 1\tn \ov{e_{\arl{s}{t}}}+\sum_{\arr{a}{b}\in E}\ov{e_{\arl{a}{b}}}\tn [\ov{e_{\arl{s}{t}}},f_{b} ] =1\tn \ov{e_{\arl{s}{t}}} + \sum_{\arr{a}{b}\in E}\ov{e_{\arl{a}{b}}}\tn f_{a} \sbt \ov{e_{\arl{s}{t}}} \sbt f_{b}  
\end{align*}
for any $\arr{s}{t}\in E$. Since $\Delta ( f_{p}\sbt \ov{f_{q}}) = f_{p}\otimes \ov{f_{q}}$, we can describe the coalgebra and counit action on the alternative generators of $H\XA^{1}$.
\begin{corollary}\label{CDiaBialg} The bialgebroid structure of $H\XA^{1}$, when viewed as a quotient of $\K D^{\mathrm{db}}$, is defined as follows:
\begin{align}
\Delta (\points{p}{q})=\sum_{i\in V}\points{p}{i}\otimes \points{i}{q},& \quad  \Delta\left((\genleft{a}{b}{c}{d})\right)=\sum_{\arr{s}{t}\in E}(\genleft{a}{b}{s}{t})\otimes (\genleft{s}{t}{c}{d}), \quad  \Delta\left((\genright{a}{b}{c}{d})\right)=\sum_{\arr{s}{t}\in E}(\genright{a}{b}{s}{t})\otimes (\genright{s}{t}{c}{d})
\\& \Delta\left((\genlefty{a}{b}{c})\right) = \sum_{\arr{s}{t}\in E}(\genleft{a}{b}{s}{t})\otimes (\genlefty{s}{t}{c}) + \sum_{t\in V}(\genlefty{a}{b}{t})\otimes \points{t}{c}
\\ &\Delta\left((\genrighty{a}{b}{c})\right)=\sum_{\arr{s}{t}\in E}(\genrighty{s}{t}{c})\otimes (\genright{s}{t}{a}{b})  + \sum_{t\in V}\points{c}{t}\otimes (\genrighty{a}{b}{t})
\\ \epsilon (\points{p}{q})= \delta_{p,q}f_{p} ,  \ \ \epsilon &\left((\genleft{a}{b}{c}{d})\right) = \delta_{a,c}.\delta_{b,d}. f_{b} , \ \ \epsilon \left((\genright{a}{b}{c}{d})\right)= \delta_{a,c}.\delta_{b,d}. f_{a} , \ \ \epsilon \left((\genlefty{a}{b}{c})\right)= 0 = \left((\genrighty{a}{b}{c})\right)
\end{align}
\end{corollary}
\begin{proof} We will present the relevant calculations for $(\genleft{a}{b}{c}{d})$ and $(\genlefty{a}{b}{c})$ and leave the remaining identities to the reader. 
\begin{align*}
\Delta\left((\genleft{a}{b}{c}{d})\right)=& \Delta\left( - \ov{f_{d}}\sbt e_{\arl{a}{b}} \sbt \ov{f_{c}}\right)=  \sum_{\arr{s}{t}\in E}\ov{f_{t}}\sbt e_{\arl{a}{b}}\sbt \ov{f_{s}}\otimes  \ov{f_{d}}\sbt e_{\arl{s}{t}} \sbt \ov{f_{c}} - e_{\arl{a}{b}}\otimes \ov{f_{d}}\sbt 1\sbt \ov{f_{c}} 
\\=& \sum_{\arr{s}{t}\in E}(\genleft{a}{b}{s}{t})\otimes (\genleft{s}{t}{c}{d})
\\ \Delta\left((\genlefty{a}{b}{c})\right)=& \Delta\left(  \ov{f_{c}}\sbt e_{\arl{a}{b}} \sbt \ov{f_{c}}- f_{b}\sbt \ov{f_{c}}\right)= e_{\arl{a}{b}}\otimes \ov{f_{c}}- \sum_{\arr{s}{t}\in E}\ov{f_{t}}\sbt e_{\arl{a}{b}}\sbt \ov{f_{s}}\otimes  \ov{f_{c}}\sbt e_{\arl{s}{t}} \sbt \ov{f_{c}}  - f_{b}\otimes \ov{f_{c}}
\\= & \sum_{\arr{s}{t}\in E} \ov{f_{t}}\sbt e_{\arl{a}{b}}\sbt \ov{f_{s}} \otimes f_{t}\sbt \ov{f_{c}}+\sum_{t\in V} \ov{f_{t}}\sbt e_{\arl{a}{b}}\sbt \ov{f_{t}} \otimes f_{t}\sbt \ov{f_{c}} - \sum_{t\in V} f_{b}\sbt \ov{f_{t}}\otimes f_{t}\sbt \ov{f_{c}} 
\\& - \sum_{\arr{s}{t}\in E}\ov{f_{t}}\sbt e_{\arl{a}{b}}\sbt \ov{f_{s}}\otimes  \ov{f_{c}}\sbt e_{\arl{s}{t}} \sbt \ov{f_{c}}
\\= & -\sum_{\arr{s}{t}\in E}\ov{f_{t}}\sbt e_{\arl{a}{b}}\sbt \ov{f_{s}}\otimes  (\genlefty{s}{t}{c}) + \sum_{t\in V}(\genlefty{a}{b}{t})\otimes f_{t}\sbt \ov{f_{c}} = \sum_{\arr{s}{t}\in E}(\genleft{a}{b}{s}{t})\otimes (\genlefty{s}{t}{c}) + \sum_{t\in V}(\genlefty{a}{b}{t})\otimes \points{t}{c}
\end{align*}
Similarly, we can describe the action of the counit: 
\begin{align*}
\epsilon\left((\genleft{a}{b}{c}{d})\right)=& \epsilon\left( - \ov{f_{d}}\sbt e_{\arl{a}{b}} \sbt \ov{f_{c}}\right)=-\epsilon\left(  \ov{f_{d}}\sbt e_{\arl{a}{b}} \sbt f_{c}\right)= -\epsilon\left(  \ov{f_{d}}\sbt \epsilon ( e_{\arl{a}{b}} \sbt f_{c})\right)
\\ =& -\epsilon\left( \ov{f_{d}}\sbt \epsilon \big( \delta_{a,c}( e_{\arl{a}{b}}-_{b})+  \delta_{b,c}f_{b}\big)\right)= -\epsilon\left( \delta_{b,c}\ov{f_{d}}\sbt  f_{b} -  \delta_{a,c}\ov{f_{d}}\sbt f_{b}\right)
\\=&\delta_{a,c}.\delta_{b,d}. f_{b}- \delta_{b,c}.\delta_{b,d}  f_{b}= \delta_{a,c}.\delta_{b,d}. f_{b} 
\\\epsilon\left((\genlefty{a}{b}{c})\right)=&\epsilon \left(  \ov{f_{c}}\sbt e_{\arl{a}{b}} \sbt \ov{f_{c}}- f_{b}\sbt \ov{f_{c}}\right)= \epsilon\left( \ov{f_{c}}\sbt \epsilon \big( \delta_{a,c}( e_{\arl{a}{b}}-f_{b})+  \delta_{b,c}f_{b}\big)\right)-\delta_{c,b} f_{b}
\\=&   \delta_{b,c}f_{b}-  \delta_{b,c}\delta_{a,c}f_{b}-\delta_{c,b} f_{b}= 0 \qedhere
\end{align*}
\end{proof}
As observed in Example 4.17 of \cite{ghobadi2020hopf}, not only is $H\XA^{1}$ a Hopf algebroid, but it also admits an antipode. This is a consequence of Theorem 4.13 of \cite{ghobadi2020hopf} which we review at the end of Section \ref{SBialSurj}. In the case of the digraph calculus $\Upsilon (e_{\arl{s}{t}}) = f_{s}-f_{t}$ and the antipode is defined on elements $e_{\arl{s}{t}}$ by $S(e_{\arl{s}{t}})=-\ov{e_{\arl{s}{t}}}-\ov{\Upsilon (e_{\arl{s}{t}}) }= -\ov{e_{\arl{s}{t}}}- \ov{f_{s}} +\ov{ f_{t}}$. We now use our description of the action of $S$ from Section \ref{SBialSurj} and simplify $S (\ov{e_{\arl{s}{t}}})$  first by observing that $b_{i}x_{j}= -\sum_{\arr{a}{b}\in E}f_{a}.e_{\arl{a}{b}}=0 \in \XA^{1}$, and secondly using the expansion of relation \eqref{EqSurjHpf2} in the proof of Theorem \ref{THDig}: 
\begin{align*}
S (\ov{e_{\arl{s}{t}}})& = -\sum_{\arr{a}{b}\in E} \ov{e_{\arl{a}{b}}} \sbt  e_{\arl{s}{t}}\sbt (-\ov{f_{a}})   - \sum_{\arr{a}{b}\in E}\ov{\Upsilon (e_{\arl{a}{b}}) }\sbt  e_{\arl{s}{t}}\sbt (-\ov{f_{a}}) 
\\&= \sum_{\arr{a}{b}\in E} f_{t}\sbt \ov{e_{\arl{a}{b}}} \sbt  e_{\arl{s}{t}}\sbt \ov{f_{a}}- f_{s}  + \sum_{\arr{a}{b}\in E}\ov{f_{a} }\sbt  e_{\arl{s}{t}}\sbt \ov{f_{a}}- \sum_{\arr{a}{b}\in E}\ov{f_{b} }\sbt  e_{\arl{s}{t}}\sbt \ov{f_{a}}
\end{align*}
Hence, we can describe the action of the antipode on the alternative set of generators: 
\begin{corollary}\label{CDigAnti} The Hopf algebroid $H\XA^{1}$ admits an invertible antipode and when viewed as a quotient of $\K D^{\mathrm{db}}$, the action of the antipode is defined by
\begin{align}
S(\points{p}{q})=\points{q}{p}, \quad S\left((\genleft{a}{b}{c}{d})\right)= (\genright{c}{d}{a}{b}), \quad S\left((\genright{a}{b}{c}{d})\right)= (\genleft{c}{d}{a}{b})
\\ S\left((\genlefty{a}{b}{q})\right)=- (\genrighty{a}{b}{q}), \quad S\left((\genrighty{c}{d}{p})\right)= -\sum_{t: \arr{p}{t} \in E}(\genrighty{p}{t}{d})\sbt (\genleft{c}{d}{p}{t})
\end{align}
for arbitrary points $p,q\in V$ and edges $\arr{a}{b}, \arr{c}{d}\in E$. 
\end{corollary} 
\begin{proof} We will first present the relevant calculations for $(\genleft{a}{b}{c}{d})$ and $(\genlefty{a}{b}{c})$:
\begin{align*}
S\left((\genleft{a}{b}{c}{d})\right)=& S\left( - \ov{f_{d}}\sbt e_{\arl{a}{b}} \sbt \ov{f_{c}}\right)=- S\left( \ov{f_{c}}\right) \sbt S\left( e_{\arl{a}{b}}\right) \sbt S\left( \ov{f_{d}}\right) = f_{c}\sbt\left(  \ov{e_{\arl{a}{b}}} +\ov{f_{a}} -\ov{ f_{b}}\right)\sbt f_{d}
\\=&f_{c}\sbt \ov{e_{\arl{a}{b}}}\sbt f_{d}+\delta_{c,d}( \ov{f_{a}} -\ov{ f_{b}}) \sbt f_{c} = f_{c}\sbt \ov{e_{\arl{a}{b}}}\sbt f_{d} = (\genright{c}{d}{a}{b})
\\S\left((\genlefty{a}{b}{c})\right)=& S\left(  \ov{f_{c}}\sbt e_{\arl{a}{b}} \sbt \ov{f_{c}}- f_{b}\sbt \ov{f_{c}}\right)=S\left( \ov{f_{c}}\right) \sbt S\left( e_{\arl{a}{b}}\right) \sbt S\left( \ov{f_{c}}\right)- S\left( \ov{f_{c}}\right)\sbt S\left( f_{b}\right)
\\=& -f_{c}\sbt\left(  \ov{e_{\arl{a}{b}}} +\ov{f_{a}} -\ov{ f_{b}}\right)\sbt f_{c} -\ov{ f_{b}}\sbt f_{c} = -f_{c}\sbt \ov{e_{\arl{a}{b}}}\sbt f_{c}  -\ov{f_{a}}\sbt f_{c} =  -(\genrighty{a}{b}{c})
\end{align*}
For the other types of generators, we use our expansion of $S (\ov{e_{\arl{s}{t}}})$, described before the corollary: 
\begin{align*}
S\left((\genright{a}{b}{c}{d})\right)=& S\left( f_{a}\sbt \ov{e_{\arl{c}{d}}} \sbt f_{b}\right)= S\left( f_{b}\right) \sbt S\left(\ov{e_{\arl{c}{d}}}\right) \sbt S\left( f_{a}\right) = 
\\=&\ov{f_{b}}\sbt \left( \sum_{\arr{s}{t}\in E} f_{d}\sbt \ov{e_{\arl{s}{t}}} \sbt  e_{\arl{c}{d}}\sbt \ov{f_{s}}- f_{c}  + \sum_{\arr{s}{t}\in E}\ov{f_{s} }\sbt  e_{\arl{c}{d}}\sbt \ov{f_{s}}- \sum_{\arr{s}{t}\in E}\ov{f_{t} }\sbt  e_{\arl{c}{d}}\sbt \ov{f_{s}}\right)\sbt \ov{f_{a}}
\\=&\delta_{a,b} \left( \sum_{t:\arr{b}{t}\in E} f_{d}\sbt \ov{e_{\arl{b}{t}}} \sbt  e_{\arl{c}{d}}\sbt \ov{f_{a}}- f_{c}\sbt \ov{f_{a}}  + \ov{f_{a} }\sbt  e_{\arl{c}{d}}\sbt \ov{f_{a}}\right) - \ov{f_{b} }\sbt  e_{\arl{c}{d}}\sbt \ov{f_{a}} = (\genleft{c}{d}{a}{b})
\\S\left((\genrighty{c}{d}{a})\right)=&  S\left( f_{a}\sbt \ov{e_{\arl{c}{d}}} \sbt f_{a} +\ov{f_{c}}\sbt f_{a}\right)= S\left( f_{b}\right) \sbt S\left(\ov{e_{\arl{c}{d}}}\right) \sbt S\left( f_{a}\right) + S\left( f_{a}\right) \sbt S\left( \ov{f_{c}}\right) 
\\=&\ov{f_{a}}\sbt \left( \sum_{\arr{s}{t}\in E} f_{d}\sbt \ov{e_{\arl{s}{t}}} \sbt  e_{\arl{c}{d}}\sbt \ov{f_{s}}- f_{c}  + \sum_{\arr{s}{t}\in E}\ov{f_{s} }\sbt  e_{\arl{c}{d}}\sbt \ov{f_{s}}- \sum_{\arr{s}{t}\in E}\ov{f_{t} }\sbt  e_{\arl{c}{d}}\sbt \ov{f_{s}} \right)\sbt \ov{f_{a}}+ \ov{f_{a}}\sbt f_{c}
\\=& \sum_{t: \arr{a}{t}\in E} f_{d}\sbt \ov{e_{\arl{a}{t}}} \sbt  e_{\arl{c}{d}}\sbt \ov{f_{a}} + \sum_{t:\arr{a}{t}\in E}\ov{f_{a} }\sbt  e_{\arl{c}{d}}\sbt \ov{f_{a}} = \sum_{t: \arr{a}{t}\in E} ( f_{d}\sbt \ov{e_{\arl{a}{t}}}\sbt f_{d} + \ov{f_{a} }\sbt f_{d}) \sbt e_{\arl{c}{d}}\sbt \ov{f_{a}}
\\& =\sum_{t: \arr{a}{t} \in E}(\genrighty{a}{t}{d})\sbt \left ( (\genlefty{c}{d}{a})-\sum_{f:\arr{a}{f}\in E}(\genleft{c}{d}{a}{f}) +\points{d}{a} \right)= -\sum_{t: \arr{a}{t} \in E}(\genrighty{a}{t}{d})\sbt (\genleft{c}{d}{a}{t})\qedhere
\end{align*}
\end{proof}
While the antipode appears to be involutive on generators of the form $(\genright{a}{b}{c}{d})$ and $(\genleft{a}{b}{c}{d})$, it is easy to see that $S^{-1}\neq S$ since 
\begin{equation}
S^{-1}  \left( (\genrighty{a}{b}{c})\right) = -(\genlefty{a}{b}{c}),\quad S^{-1}\left((\genlefty{a}{b}{c}) \right) = -\sum_{s: \arr{s}{c} \in E}(\genlefty{s}{c}{a})\sbt (\genright{s}{c}{a}{b})
\end{equation} 
We have omitted the calculations for the action $S^{-1}$ since we will not be using them. 
\subsection{Flat Connections on Digraph}\label{SMaxProlong} 
In this section we construct $\ct{D}\XA$ for the digraph calculus with a specific choice of $\Omega^{2}$. Throughout this section $\arrtwo{a}{b}{c}$ will denote a path of length two in the digraph i.e. two successive arrows $\arr{a}{b},\arr{b}{c} \in E$ and $E_{2}$ will denote the set of all paths of length two in the digraph. 

In order to discuss flat connections over the digraph calculus, we must introduce the space of 2-forms extending $(A,d,\Omega^{1})$. Noting that $\Omega^{1}\otimes \Omega^{1}$ is spanned by elements $\omega_{\arr{a}{b}}\tn \omega_{\arr{b}{c}}$ corresponding to paths of length two, we defined $\Omega^{2}_{\mathrm{top}}$ as the quotient of $\Omega^{1}\otimes \Omega^{1}$ by the submodule 
\begin{equation}\label{EqMID} 
\ct{M}_{\mathrm{top}} =\left\{ \sum_{b:\arrtwo{a}{b}{c}\in E_{2}} \omega_{\arr{a}{b}}\otimes \omega_{\arr{b}{c}} \middle\vert a,c\in V ,\ a\nrightarrow c \right\}
\end{equation}
with the differential $d:\Omega^{1}\rightarrow \Omega^{2}_{\mathrm{top}}$ defined by 
\begin{equation}\label{EqDiffDigraph}
d( \omega_{\arr{a}{b}})=\sum_{t:\arr{b}{t}\in E}\omega_{\arr{a}{b}}\wedge \omega_{\arr{b}{t}} + \sum_{s:\arr{s}{a}\in E} \omega_{\arr{s}{a}}\wedge \omega_{\arr{a}{b}} - \sum_{i:\arrtwo{a}{i}{b}\in E_{2}} \omega_{\arr{a}{i}}\wedge \omega_{\arr{i}{b}}
\end{equation}
where $\wedge : \Omega^{1}\otimes \Omega^{1} \twoheadrightarrow \Omega^{2}$ denotes the natural projection. 

\begin{lemma} The differential described above extends the first-order calculus of $(A,\Omega^{1})$. 
\end{lemma} 
\begin{proof} First we show that $d^{2}=0$ and $d$ is in fact a differential. 
\begin{align*}
d (d f_{p}) = & d\left(\sum_{s:\arr{s}{p}\in E} \omega_{\arr{s}{p}} - \sum_{t:\arr{p}{t}\in E} \omega_{\arr{p}{t}}\right)  
\\= & \sum_{s,t:\arrtwo{s}{p}{t}\in E_{2}}\omega_{\arr{s}{p}}\wedge \omega_{\arr{p}{t}} + \sum_{s,c:\arrtwo{c}{s}{p}\in E_{2}} \omega_{\arr{c}{s}}\wedge \omega_{\arr{s}{p}} - \sum_{\substack{s,i:\arrtwo{s}{i}{p}\in E_{2}\\\text{where}\ \arr{s}{p}\in E}} \omega_{\arr{s}{i}}\wedge \omega_{\arr{i}{p}} 
\\ & -\sum_{t,d:\arrtwo{p}{t}{d}\in E_{2}}\omega_{\arr{p}{t}}\wedge \omega_{\arr{t}{d}} -\sum_{s,t:\arrtwo{s}{p}{t}\in E_{2}} \omega_{\arr{s}{p}}\wedge \omega_{\arr{p}{t}} + \sum_{\substack{t,i:\arrtwo{p}{i}{t}\in E_{2}\\\text{where}\ \arr{p}{t}\in E}} \omega_{\arr{p}{i}}\wedge \omega_{\arr{i}{t}} 
\\=&\sum_{\substack{s,c:\arrtwo{c}{s}{p}\in E_{2}\\ \text{where}\ c\nrightarrow p}} \omega_{\arr{c}{s}}\wedge \omega_{\arr{s}{p}} -\sum_{\substack{t,d:\arrtwo{p}{t}{d}\in E_{2}\\\text{where}\ p\nrightarrow d}}\omega_{\arr{p}{t}}\wedge \omega_{\arr{t}{d}}= 0 
\end{align*}
Now we show that $d$ is compatible with the $A$-bimodule structure of $\Omega^{1}$. For $p\in V$ and $\arr{a}{b}\in E$ we observe that 
\begin{align*}
d( f_{p}. \omega_{\arr{a}{b}}) &= \delta_{a,p}.\left( \sum_{t:\arr{b}{t}\in E}\omega_{\arr{a}{b}}\wedge \omega_{\arr{b}{t}} + \sum_{s:\arr{s}{a}\in E} \omega_{\arr{s}{a}}\wedge \omega_{\arr{a}{b}} - \sum_{i:\arrtwo{a}{i}{b}\in E_{2}} \omega_{\arr{a}{i}}\wedge \omega_{\arr{i}{b}}\right) 
\\ df_{p}\wedge \omega_{\arr{a}{b}} &= \delta_{a,p}\sum_{s:\arr{s}{p}\in E} \omega_{s\rightarrow p}\wedge\omega_{\arr{a}{b}}  - \delta_{\arr{p}{a}}. \omega_{\arr{p}{a}}\wedge \omega_{\arr{a}{b}}
\\ f_{p}.d(\omega_{\arr{a}{b}})&= \delta_{a,p}\left( \sum_{t:\arr{b}{t}\in E}\omega_{\arr{a}{b}}\wedge \omega_{\arr{b}{t}}- \sum_{i:\arrtwo{a}{i}{b}\in E_{2}} \omega_{\arr{a}{i}}\wedge \omega_{\arr{i}{b}}\right)  + \delta_{\arr{p}{a}} .\omega_{\arr{p}{a}}\wedge \omega_{\arr{a}{b}} 
\end{align*}
Hence, $d( f_{p}. \omega_{\arr{a}{b}})= df_{p} \wedge \omega_{\arr{a}{b}} + f_{p}. d(\omega_{\arr{a}{b}})$ as required. We leave the verification of $d( \omega_{\arr{a}{b}}.f_{p})=  d(\omega_{\arr{a}{b}}).f_{p}- \omega_{\arr{a}{b}}\wedge df_{p} $ to the reader. 
\end{proof}
\begin{rmk}\label{ROmega2s}\rm There are several choices for $\Omega^{2}$ when extending the digraph differential calculus. The maximal prolongation described in Proposition 1.40 of \cite{beggs2020quantum}, sets terms $ \ct{Q}(a,c):= \sum_{b:\arrtwo{a}{b}{c}\in E_{2}} \omega_{\arr{a}{b}}\otimes \omega_{\arr{b}{c}}$ to zero where $a\neq c$ and $a\nrightarrow c$. One obtains our choice of $\Omega^{2}_{\mathrm{top}}$ by further quotienting out terms $\ct{Q}(a,c)$ for the case of $a=c$. There are two other choices of $\Omega^{2}$ which will keep the calculus inner, see Definition 1.3 \cite{beggs2014noncommutative}, where one further quotients out $\ct{Q}(a,c)$ for $\arr{a}{c}$, or for both cases of $a=c$ and $\arr{a}{c}$. The latter, which we denote by $\Omega^{2}_{\mathrm{min}}$, is the choice made to construct $\ct{D}\XA$ for a quiver in Example 5.7 in \cite{ghobadi2020hopf}. Another important choice of $\Omega^{2}$ is for Cayley graphs defined by a suitable subset of a group. In this case the calculus corresponds to a bicovariant calculus over a Hopf algebra and there's a natural choice for $\Omega^{2}$ by Woronowicz's work \cite{woronowicz1989differential}, which will be a quotient of $\Omega^{2}_{\mathrm{min}}$.
\end{rmk} 
Before constructing $\ct{D}\XA$, we must discuss $\XA^{2}$ and its duality morphisms. Notice that $\XA^{2}$ is the quotient of $\XA^{1}\otimes \XA^{1}$ by the corresponding submodule spanned by terms $\sum_{b:\arrtwo{a}{b}{c}\in E_{2}} e_{\arl{b}{c}}\tn e_{\arl{a}{b}}$ where $a,c\in V$ and $ a\nrightarrow c $. We denote the image of $e_{\arl{b}{c}}\tn e_{\arl{a}{b}}$ in $\XA^{2}$ by $e_{ \arltwo{a}{b}{c}}$. In order to define duality morphisms for $\XA^{2}$ and $\Omega^{2}$, we must choose a distinguished vertex $\dist{a}{c}\in \{b\in V\mid \arrtwo{a}{b}{c}\in E_{2} \} $, for every pair of vertices $a,c\in V$ where $ a\nrightarrow c $. Having made such a choice, we denote $E_{2}\setminus \{\arrtwo{a}{\dist{a}{c}}{c}\mid a,c\in V, \ a\nrightarrow c\}$ by $\mathbf{E}_{2}$ and observe that the subsets $\{\omega_{\arr{a}{b}}\wedge \omega_{\arr{b}{c}}\mid \arrtwo{a}{b}{c}\in \mathbf{E}_{2}\}$ and $\{ e_{\arltwo{a}{b}{c}}\mid \arrtwo{a}{b}{c}\in \mathbf{E}_{2}\}$ each form a basis for $\Omega^{2}$ and $\XA^{2}$, respectively. With this choice of basis we have that $\omega_{\arr{a}{\dist{a}{c}}}\wedge \omega_{\arr{\dist{a}{c}}{c}} = -\sum_{b: \arrtwo{a}{b}{c}\in \mathbf{E}_{2}}\omega_{\arr{a}{b}}\wedge \omega_{\arr{b}{c}}$ and $e_{\arltwo{a}{\dist{a}{c}}{c}} = -\sum_{b: \arrtwo{a}{b}{c}\in \mathbf{E}_{2}}e_{\arltwo{a}{b}{c}}$. Hence, we can define a quadruple of duality morphisms: 
\begin{align*}
\cvl^{2}(1)&= \sum _{\arrtwo{a}{b}{c}\in \mathbf{E}_{2}} \omega_{\arr{a}{b}}\wedge \omega_{\arr{b}{c}}\otimes e_{\arltwo{a}{b}{c}},\quad \evl^{2}(e_{\arltwo{x}{y}{z}}\otimes \omega_{\arr{p}{q}}\wedge \omega_{\arr{q}{r}}) = \delta_{x,p}.\delta_{y,q}.\delta_{z,r} f_{z}
\\\cvr^{2}(1)&= \sum _{\arrtwo{a}{b}{c}\in \mathbf{E}_{2}} e_{\arltwo{a}{b}{c}}\otimes \omega_{\arr{a}{b}}\wedge \omega_{\arr{b}{c}},\quad \evr^{2}(\omega_{\arr{p}{q}}\wedge \omega_{\arr{q}{r}}\otimes e_{\arltwo{x}{y}{z}}) = \delta_{p,x}.\delta_{q,y}.\delta_{r,z} f_{x}
\end{align*}
where $\arrtwo{p}{q}{r}, \arrtwo{x}{y}{z}\in \mathbf{E}_{2}$. Hence, we can easily check that the differential structure is pivotal by writing out equation \eqref{EqPivWedge} for a basis element $e_{\arltwo{p}{q}{r}}\in \XA^{2}$: 
\begin{align*}
\sum_{\arr{a}{b},\arr{c}{d}\in E} &\evl^{2} (e_{\arltwo{p}{q}{r}} \tn \omega_{\arr{a}{b}}\wedge \omega_{\arr{c}{d}}) e_{\arl{c}{d}}\otimes e_{\arl{a}{b}} = (\delta_{b,q}- \delta_{p\nrightarrow r}. \delta_{\dist{p}{r},b}) f_{r}.  e_{\arl{b}{r}}\otimes e_{\arl{p}{b}}
\\ = &(\delta_{b,q}- \delta_{p\nrightarrow r}. \delta_{\dist{p}{r},b})  e_{\arl{b}{r}}\otimes e_{\arl{p}{b}} =   e_{\arl{b}{r}}\otimes e_{\arl{p}{b}}. f_{p}(\delta_{b,q}- \delta_{p\nrightarrow r}. \delta_{\dist{p}{r},b})
\\ =& \sum_{\arr{a}{b},\arr{c}{d}\in E} e_{\arl{c}{d}}\otimes e_{\arl{a}{b}} \evr^{2}(\omega_{\arr{a}{b}}\wedge \omega_{\arr{c}{d}} \otimes e_{\arltwo{p}{q}{r}} ) 
\end{align*}
Since $\wedge: \Omega^{1}\tn \Omega^{1}\rightarrow \Omega^{2}$ is surjective, we can construct $\ct{D}\XA$ as a quotient of $H\XA^{1}$ by Corollary \ref{CDXA}. Before applying this result, we will introduce some additional notation. We will say a triple of vertices $(p,q,r)$ form a \emph{triangle} in the digraph if $\arr{p}{q}\in E$ and $\arrtwo{p}{q}{r}\in E_{2}$, and a quadruple of vertices $(p,q,q',r)$ form a \emph{square} in the digraph if $\arrtwo{p}{q}{r}\in E_{2}$ and $\arrtwo{p}{q'}{r}\in E_{2}$. Let us adapt the following notation
\begin{align*} \mathbf{P}(p,q,r;a,b):= &\sum_{i:\arrtwo{a}{i}{b}\in E_{2}}(\genleft{q}{r}{i}{b})\sbt (\genleft{p}{q}{a}{i})- \delta_{\arr{a}{b}}.\left( (\genlefty{q}{r}{b})\sbt (\genleft{p}{q}{a}{b})+ (\genleft{q}{r}{a}{b})\sbt (\genlefty{p}{q}{a})\right)\in H\XA
\\ \mathbf{Q}(p,q,r;a,b):=& \sum_{i:\arrtwo{a}{i}{b}\in E_{2}} (\genright{a}{i}{p}{q})\sbt (\genright{i}{b}{q}{r})+ \delta_{\arr{a}{b}}.\left(  (\genright{a}{b}{p}{q})\sbt (\genrighty{q}{r}{b})+  (\genrighty{p}{q}{a})\sbt (\genright{a}{b}{q}{r})\right)\in H\XA
\end{align*} 
for every pair of vertices $a,b\in V$ and path $\arrtwo{p}{q}{r}\in E$.
\begin{thm}\label{TDXDigraph} The Hopf algebroid $\ct{D}\XA$ for the differential calculus of a digraph $D=(V,E)$ with the choice of $\Omega^{2}_{\mathrm{top}}$, is the quotient of $\K D^{\mathrm{db}}$ by the relations observed in Theorem \ref{THDig} and the additional relations: 
\begin{itemize}
\item For any square $(p,q,q',r)$ in the digraph, we impose the following relations 
\begin{align}
(\genlefty{q}{r}{c})\sbt (\genlefty{p}{q}{c})&=(\genlefty{q'}{r}{c})\sbt (\genlefty{p}{q'}{c}) \label{EqSqr1}
\\  (\genrighty{p}{q}{c})\sbt (\genrighty{q}{r}{c})&= (\genrighty{p}{q'}{c}) \sbt (\genrighty{q'}{r}{c})\label{EqSqr1'}
\\ \mathbf{P}(p,q,r;a,b)&= \mathbf{P}(p,q',r;a,b)\label{EqSqr2}
\\ \mathbf{Q}(p,q,r;a,b)&= \mathbf{Q}(p,q',r;a,b)\label{EqSqr2'}
\end{align}
for arbitrary vertices $a,b,c\in V$
\item For any triangle $(p,q,r)$ in the digraph, we impose the following relations 
\begin{align}
(\genlefty{q}{r}{c})\sbt (\genlefty{p}{q}{c})&= - (\genlefty{p}{r}{c}) \label{EqTri1}
\\(\genrighty{p}{q}{c})\sbt (\genrighty{q}{r}{c})&=  (\genrighty{p}{r}{c}) \label{EqTri1'}
\\\mathbf{P}(p,q,r;a,b)&= \delta_{\arr{a}{b}}(\genleft{p}{r}{a}{b}) \label{EqTri2}
\\\mathbf{Q}(p,q,r;a,b)&= \delta_{\arr{a}{b}}(\genright{a}{b}{p}{r}) \label{EqTri2'}
\end{align}
for arbitrary vertices $a,b,c\in V$
\end{itemize}
\end{thm}
\begin{proof}
Note that $\ker (\wedge) = \ct{M}_{\mathrm{top}}$ is spanned by $\sum_{x:\arrtwo{x}{y}{z}\in E_{2}} f_{x}df_{y}\wedge df_{z}$, where $x,z\in V$ such that $x\nrightarrow z$. We first expand relation \eqref{EqSujExt1} for an arbitrary basis element $e_{\arltwo{p}{q}{r}}\in \XA^{2}$:
\begin{align*}
0=&\sum_{\substack{\arr{a}{b},\\ \arr{c}{d}}\in E}\left( \sum_{y:\arrtwo{x}{y}{z}\in E_{2}}\evl^{2}(e_{\arltwo{p}{q}{r}}\otimes \omega_{\arr{a}{b}}\wedge \omega_{\arr{c}{d}}) \sbt [e_{\arl{c}{d}},\ov{f_{z}}] \sbt [e_{\arl{a}{b}},\ov{f_{y}}]\sbt\ov{f_{x}} \right)
\\=&\sum_{\arrtwo{a}{b}{d}\in E_{2}}\left( \sum_{y:\arrtwo{x}{y}{z}\in E_{2}}\evl^{2}(e_{\arltwo{p}{q}{r}}\otimes \omega_{\arr{a}{b}}\wedge \omega_{\arr{b}{d}}) \sbt \ov{f_{z}}\sbt e_{\arl{b}{d}}\sbt \ov{f_{y}}\sbt e_{\arl{a}{b}}\sbt\ov{f_{x}} \right)
\\= &\sum_{y:\arrtwo{x}{y}{z}\in E_{2}} \ov{f_{z}}\sbt e_{\arl{q}{r}}\sbt \ov{f_{y}}\sbt e_{\arl{p}{q}}\sbt\ov{f_{x}} - \delta_{p\nrightarrow r} \sum_{y:\arrtwo{x}{y}{z}\in E_{2}} \ov{f_{z}}\sbt e_{\arl{\dist{p}{r}}{r}}\sbt \ov{f_{y}}\sbt e_{\arl{p}{\dist{p}{r}}}\sbt\ov{f_{x}}
\end{align*}
Changing generators, we obtain two sets of relations. Firstly, we notice that for a pair of vertices $x,z\in V$ such that $x\nrightarrow z$ and a path $\arrtwo{p}{q}{r}$, where $p\nrightarrow r$ we have that 
\begin{equation}\label{EqInProofSurj1}
\sum_{y:\arrtwo{x}{y}{z}\in E_{2}} (\genleft{q}{r}{y}{z})\sbt (\genleft{p}{q}{x}{y})= \sum_{y:\arrtwo{x}{y}{z}\in E_{2}} (\genleft{\dist{p}{r}}{r}{y}{z})\sbt (\genleft{p}{\dist{p}{r}}{x}{y})
\end{equation}
This demonstrates that $\mathbf{P}(p,q,r,a,b)$ is independent of $q$ when $a\nrightarrow b$, and proves relation \eqref{EqSqr2} when $p\nrightarrow r$. Secondly, for a pair of vertices $x,z\in V$ such that $x\nrightarrow z$ and a triangle $(p,q,r)$, we have that 
\begin{equation}\label{EqInProofSurj2}
\sum_{y:\arrtwo{x}{y}{z}\in E_{2}} (\genleft{q}{r}{y}{z})\sbt (\genleft{p}{q}{x}{y})= 0
\end{equation}
which demonstrates relation \eqref{EqTri2} for $a\nrightarrow b$. 

In a symmetric manner, relation \eqref{EqSujExt2} when written in terms of the alternative set of generators, becomes equivalent to relations \eqref{EqSqr2'} with $p\nrightarrow r$ and $a\nrightarrow b$, and \eqref{EqTri2'} with $a\nrightarrow b$. 

Now we look at relation \eqref{EqFlat} for an arbitrary basis element $e_{\arltwo{p}{q}{r}}\in \XA^{2}$:
\begin{equation}\label{EqFlatInPrf}
0= \sum_{\arr{s}{t}\in E}\evl^{2}(e_{\arltwo{p}{q}{r}}\otimes d\omega_{\arr{s}{t}})\sbt e_{\arl{s}{t}} - \sum_{\arrtwo{a}{b}{c}\in E_{2}} \evl^{2}(e_{\arltwo{p}{q}{r}}\otimes \omega_{\arr{a}{b}}\wedge\omega_{\arr{b}{c}}) \sbt e_{\arl{b}{c}}\sbt e_{\arl{a}{b}} 
\end{equation}
We will simplify \eqref{EqFlatInPrf} in two cases. First if $\arr{p}{r}\in E$: 
\begin{align*}
0=& \sum_{\arr{s}{t}\in E}(\delta_{s,q}.\delta_{r,t} + \delta_{s,p}.\delta_{t,q}- \delta_{p,s}\delta_{r,t})f_{r}\sbt e_{\arl{s}{t}}  -  \sum_{\arrtwo{a}{b}{c}\in E_{2}}\delta_{a,p}.\delta_{b,q}.\delta_{c,r}f_{c} \sbt e_{\arl{b}{c}}\sbt e_{\arl{a}{b}} 
\\=& e_{\arl{q}{r}}+ f_{r}\sbt e_{\arl{p}{q}} - e_{\arl{p}{r}} - e_{\arl{q}{r}}\sbt e_{\arl{p}{q}} = e_{\arl{q}{r}} \sbt ( 1-  e_{\arl{p}{q}})- e_{\arl{p}{r}}
\end{align*}
Now we change generators as in the proof of Theorem \ref{THDig}: 
 \begin{align*}
0=&\left(\sum_{a\in V}(\genlefty{q}{r}{a})-\sum_{\arr{a}{b}\in E}(\genleft{q}{r}{a}{b}) +\sum_{a\in V}\points{r}{a} \right)  \sbt  \left( \sum_{c,d\in V} \points{c}{d}- \sum_{c\in V}(\genlefty{p}{q}{c})+\sum_{\arr{c}{d}\in E}(\genleft{p}{q}{c}{d}) -\sum_{c\in V}\points{q}{c} \right) 
\\&- \sum_{a\in V}(\genlefty{p}{r}{a})+\sum_{\arr{a}{b}\in E}(\genleft{p}{r}{a}{b}) -\sum_{a\in V}\points{r}{a} 
\\=&\sum_{a\in V}\points{r}{a} - \sum_{a\in V}(\genlefty{q}{r}{a})\sbt (\genlefty{p}{q}{a})+\sum_{\arr{a}{b}\in E}(\genlefty{q}{r}{b})\sbt (\genleft{p}{q}{a}{b})+ \sum_{\arr{a}{b}\in E}(\genleft{q}{r}{a}{b})\sbt (\genlefty{p}{q}{a})- \sum_{\arrtwo{a}{b}{c}\in E_{2}}(\genleft{q}{r}{b}{c})\sbt (\genleft{p}{q}{a}{b})
\\&- \sum_{a\in V}(\genlefty{p}{r}{a})+\sum_{\arr{a}{b}\in E}(\genleft{p}{r}{a}{b}) -\sum_{a\in V}\points{r}{a} 
\end{align*} 
Using \eqref{EqInProofSurj2}, we can further simplify the relation: 
\begin{align*}
0=& - \sum_{a\in V}(\genlefty{q}{r}{a})\sbt (\genlefty{p}{q}{a})+\sum_{\arr{a}{b}\in E}(\genlefty{q}{r}{b})\sbt (\genleft{p}{q}{a}{b})+ \sum_{\arr{a}{b}\in E}(\genleft{q}{r}{a}{b})\sbt (\genlefty{p}{q}{a})- \sum_{\substack{i:\arrtwo{a}{i}{b}\in E_{2}\\ \arr{a}{b}\in E}}(\genleft{q}{r}{i}{b})\sbt (\genleft{p}{q}{a}{i})
\\&- \sum_{a\in V}(\genlefty{p}{r}{a})+\sum_{\arr{a}{b}\in E}(\genleft{p}{r}{a}{b})
\end{align*} 
Hence, we observe from the path algebra structure that the relation is equivalent to the following the equalities holding for arbitrary $\arr{a}{b}\in E$ and $c\in V$: 
\begin{align*}
\sum_{i:\arrtwo{a}{i}{b}\in E_{2}}(\genleft{q}{r}{i}{b})\sbt (\genleft{p}{q}{a}{i})- (\genlefty{q}{r}{b})\sbt (\genleft{p}{q}{a}{b})- (\genleft{q}{r}{a}{b})\sbt (\genlefty{p}{q}{a})= (\genleft{p}{r}{a}{b}), \quad (\genlefty{q}{r}{c})\sbt (\genlefty{p}{q}{c})= - (\genlefty{p}{r}{c})
\end{align*}
which agree with relation \eqref{EqTri2} when $\arr{a}{b}\in E$ and \eqref{EqTri1}, respectively. 

Now we expand relation \eqref{EqFlatInPrf} in the case where ${p}\nrightarrow {r}$: 
\begin{align*}
0&= ( \delta_{s,q} - \delta_{s, \dist{p}{r}} - \delta_{\arr{p}{r}}.\delta_{s,p}) . f_{r}\sbt e_{\arl{s}{r}}   - ( \delta_{b,q} -\delta_{b, \dist{p}{r}} ). f_{r}\sbt e_{\arl{b}{r}}\sbt e_{\arl{p}{b}}
\\&= e_{\arl{q}{r}} -    e_{\arl{\dist{p}{r}}{r}} - e_{\arl{q}{r}}\sbt e_{\arl{p}{q}} + e_{\arl{\dist{p}{r}}{r}}\sbt e_{\arl{p}{\dist{p}{r}}} 
\\  &\Longleftrightarrow\  \ e_{\arl{\dist{p}{r}}{r}}- e_{\arl{\dist{p}{r}}{r}}\sbt e_{\arl{p}{\dist{p}{r}}} = e_{\arl{q}{r}} - e_{\arl{q}{r}}\sbt e_{\arl{p}{q}}
\end{align*} 
From our previous computations we know that $e_{\arl{q}{r}} - e_{\arl{q}{r}}\sbt e_{\arl{p}{q}}$ can be written as 
\begin{equation*}
\sum_{a\in V}\points{r}{a} - \sum_{a\in V}(\genlefty{q}{r}{a})\sbt (\genlefty{p}{q}{a})+\sum_{\arr{a}{b}\in E}(\genlefty{q}{r}{b})\sbt (\genleft{p}{q}{a}{b})+ \sum_{\arr{a}{b}\in E}(\genleft{q}{r}{a}{b})\sbt (\genlefty{p}{q}{a})- \sum_{\arrtwo{a}{b}{c}\in E_{2}}(\genleft{q}{r}{b}{c})\sbt (\genleft{p}{q}{a}{b})
\end{equation*} 
after changing generators. Note that \eqref{EqInProofSurj1} implies 
\begin{align*}\sum_{b:\arrtwo{a}{b}{a}\in E_{2}}(\genleft{q}{r}{b}{a})\sbt (\genleft{p}{q}{a}{b})=\sum_{b:\arrtwo{a}{b}{a}\in E_{2}}(\genleft{\dist{p}{q}}{r}{b}{a})\sbt (\genleft{p}{\dist{p}{q}}{a}{b})\end{align*}
Hence, we can expand the above equality on either side and we precisely obtain relations 
\begin{align}
 (\genlefty{q}{r}{a})\sbt (\genlefty{p}{q}{a}) = (\genlefty{\dist{p}{r}}{r}{a})\sbt (\genlefty{p}{\dist{p}{r}}{a}), \quad \mathbf{P}(p,q,r,a,b)= \mathbf{P}(p,\dist{p}{r},r,a,b)
\end{align}
which demonstrate relations \eqref{EqSqr1} and \eqref{EqSqr2} in the case where $p\nrightarrow r$. Note that relation \eqref{EqSqr1} follows automatically from \eqref{EqTri1} when $\arr{p}{r}\in E$.

At this point we have covered all relations which must be imposed on $H\XA^{1}$ to obtain $\ct{D}\XA$. Recall that \eqref{EqFlatRght} holds automatically in this case. We obtain the remaining relations in Theorem \ref{TDXDigraph} by expanding \eqref{EqFlatRght}. Note that for $e_{\arltwo{p}{q}{r}}\in \XA^{2}$, \eqref{EqFlatRght} expands as
\begin{align*}
0=&\sum_{\arr{a}{b}\in E} \ov{\evr^{2}(d\omega_{\arr{a}{b}}\otimes e_{\arltwo{p}{q}{r}}) }\sbt\ov{e_{\arl{a}{b}}}+ \sum_{\arrtwo{a}{b}{c}\in E}\ov{\evr^{2}(\omega_{\arr{a}{b}}\wedge \omega_{\arr{b}{c}}\otimes e_{\arltwo{p}{q}{r}})}\sbt \ov{e_{\arl{a}{b}}}\sbt \ov{e_{\arl{b}{c}}}
\\=&\ov{e_{\arl{p}{q}}} + \ov{e_{\arl{p}{q}}}\sbt \ov{e_{\arl{q}{r}}} - \delta_{\arr{p}{r}} \ov{e_{\arl{p}{r}}} - \delta_{p\nrightarrow r} ( \ov{e_{\arl{p}{\dist{p}{r}}}} + \ov{e_{\arl{p}{\dist{p}{r}}}}\sbt \ov{e_{\arl{\dist{p}{r}}{r}}})
\end{align*}
Additionally, we obtain the following expansion in terms of alternate generators: 
\begin{align*}
\ov{e_{\arl{p}{q}}}& + \ov{e_{\arl{p}{q}}}\sbt \ov{e_{\arl{q}{r}}}= \ov{e_{\arl{p}{q}}}( 1+ \ov{e_{\arl{p}{q}}})
\\ =&\left(\sum_{c\in V}(\genrighty{p}{q}{c})+ \sum_{\arr{a}{b}\in E}(\genright{a}{b}{p}{q}) -\sum_{c\in V}\points{c}{p}  \right) \sbt \left( \sum_{c,d\in V}\points{d}{c}+\sum_{c\in V}(\genrighty{q}{r}{c})+ \sum_{\arr{a}{b}\in E}(\genright{a}{b}{q}{r}) -\sum_{c\in V}\points{c}{q} \right) 
\\=&\sum_{c\in V}(\genrighty{p}{q}{c})\sbt (\genrighty{q}{r}{c})+ \sum_{\arr{a}{b}\in E}(\genrighty{p}{q}{a})\sbt(\genright{a}{b}{q}{r}) +\sum_{\arr{a}{b}\in E}(\genright{a}{b}{p}{q})\sbt (\genrighty{q}{r}{b})+ \sum_{\arrtwo{a}{b}{c}\in E_{2}}(\genright{a}{b}{p}{q})\sbt(\genright{b}{c}{q}{r})  -\sum_{c\in V}\points{c}{p} 
\end{align*}
From the above expansions, it is easy to check that relation \eqref{EqFlatRght} is equivalent to relations \eqref{EqSqr1'}, \eqref{EqTri1'} and the remaining cases of \eqref{EqSqr2'}, \eqref{EqTri2'}.
\end{proof}
Finally, let us comment on whether $\ct{D}\XA$ will admit an antipode. Rather than checking whether the relations in Theorem \ref{TDXDigraph} are $S$-stable, which is a lengthy task, we will apply Theorem 5.5 of \cite{ghobadi2020hopf}, which states additional conditions on $\Upsilon$ which must hold. As recalled in Section \ref{SBialSurj}, we will only need to check equation \eqref{EqUpFla} in $A$ for arbitrary $x^{2}\in \XA^{2}$. Since for the digraph calculus $\Upsilon$ is defined by $\Upsilon (e_{\arl{a}{b}})= f_{a}-f_{b}$, we can simplify equation \eqref{EqUpFla} for $e_{\arltwo{p}{q}{r}}\in \XA^{2}$ as follows: 
\begin{align*}
\sum_{\arr{a}{b}\in E} &\Upsilon (\evl^{2}(e_{\arltwo{p}{q}{r}}\tn d\omega_{\arr{a}{b}}).e_{\arl{a}{b}} ) +\sum_{\arrtwo{a}{b}{c}\in E_{2}}\Upsilon \big(\Upsilon(\evl^{2}(e_{\arltwo{p}{q}{r}}\tn \omega_{\arr{a}{b}}\wedge \omega_{\arr{b}{c}}).e_{\arl{b}{c}} ).e_{\arl{a}{b}}\big)
\\&=\Upsilon (e_{\arl{q}{r}}- \delta_{p\nrightarrow r}  e_{\arl{\dist{p}{r}}{r}}- \delta_{\arr{p}{r}} e_{\arl{p}{r}}) +\Upsilon \big(\Upsilon(e_{\arl{q}{r}} ).e_{\arl{p}{q}}- \delta_{p\nrightarrow r} \Upsilon(e_{\arl{\dist{p}{r}}{r}}).e_{\arl{p}{\dist{p}{r}}}\big)
\\&=f_{q}-f_{r} -  \delta_{p\nrightarrow r} ( f_{\dist{p}{r}} -f_{r})- \delta_{\arr{p}{r}}( f_{p}-f_{r}) +\Upsilon \big(e_{\arl{p}{q}}- \delta_{p\nrightarrow r} .e_{\arl{p}{\dist{p}{r}}}\big)
\\&=f_{q}-f_{r} +f_{p}-f_{q} -  \delta_{p\nrightarrow r} ( f_{\dist{p}{r}} -f_{r} +f_{p}-f_{\dist{p}{r}})- \delta_{\arr{p}{r}}( f_{p}-f_{r}) 
\\&=f_{p}-f_{r} -  \delta_{p\nrightarrow r} ( f_{p} -f_{r} )- \delta_{\arr{p}{r}}( f_{p}-f_{r}) =0
\end{align*}
Hence the antipode of $H\XA^{1}$, as described in Corollary \ref{CDigAnti} carries over to $\ct{D}\XA$. Note that we have done the above calculation for the calculus of an arbitrary quiver without cycles in Example 5.7 in \cite{ghobadi2020hopf}, with the choice of $\Omega^{2}_{\mathrm{max}}$. As detailed in Remark \ref{ROmega2s}, this calculation does not apply to our setting here, due to the choice of $\Omega^{2}_{\rm top}$. 
\subsection{Fundamental Groupoid of Digraph}\label{SIsoDigraph}
In this section we recover the groupoid ring, see Example \ref{ExGrpd}, of the fundamental groupoid of a digraph as the isotropy quotient of $\ct{D}\XA$ which was constructed in the previous section. 

First we consider the isotopy quotient of $H\XA^{1}$, which will again be a Hopf algebroid over the base algebra $\K (V)$. The isotopy quotient is obtained by imposing the additional relation $f_{q}=\ov{f_{q}}$ for all $q\in V$. Thereby, in terms of our path algebra presentation of $H\XA^{1}$ from Theorem \ref{THDig}, the relation imposes $\points{a}{b}=\delta_{a,b} \points{a}{a}$ for arbitrary $a,b\in V$. From the algebra structure, we conclude that 
\begin{align*}(\genlefty{a}{b}{c})=\points{a}{c}\sbt (\genlefty{a}{b}{c})\sbt \points{b}{c} =\delta_{a,c}.\delta_{b,c} (\genlefty{a}{b}{c})=0
\end{align*}
Similarly, we observe that $(\genlefty{a}{b}{c})= 0$ and $ (\genleft{a}{b}{c}{d})=\delta_{a,c}.\delta_{b,d}(\genleft{a}{b}{a}{b})$ and $ (\genright{a}{b}{c}{d})=\delta_{a,c}.\delta_{b,d} (\genright{a}{b}{a}{b})$. Hence we relabel the non-vanished generators in $\iso{H\XA^{1}}$: 
\begin{equation}
(p):= \points{p}{p}, \quad (\arl{a}{b}):=(\genleft{a}{b}{a}{b}), \quad (\arr{a}{b}) =(\genright{a}{b}{a}{b})
\end{equation}
and obtain a presentation for $\iso{H\XA^{1}}$.
\begin{corollary}\label{CIsoHXDigraph} The Hopf algebroid $\iso{H\XA^{1}}$ is isomorphic to the groupoid algebra $\K \ct{G}_{D}$, where $\ct{G}_{D}$ is the free groupoid corresponding to the digraph, where we have added an inverse for each arrow in the digraph and arrows corresponding to all possible new paths.
\end{corollary} 
\begin{proof} While relations \eqref{EqDigH3} vanish in $\iso{H\XA^{1}}$, relations \eqref{EqDigH1} and \eqref{EqDigH2}, both simply reduce to showing $(\arl{a}{b})$ and $(\arr{a}{b})$ are inverses e.g. the left hand equation in \eqref{EqDigH1}: 
\begin{align*}
\sum_{i: \arr{q}{i} \in E}(\genright{a}{b}{q}{i})\sbt (\genleft{p}{b}{q}{i}) =\delta_{i,b}\sum_{i: \arr{q}{i} \in E}(\genright{a}{b}{q}{i})\sbt (\genleft{p}{b}{q}{i})= (\genright{a}{b}{q}{b})\sbt (\genleft{p}{b}{q}{b})= \delta_{q,p}.\delta_{q,a} (\arr{a}{b})\sbt (\arl{a}{b}) = \delta_{a,p}.\delta_{a,q} (a)
\end{align*}
A similar argument on \eqref{EqDigH2} shows that $(\arl{a}{b})\sbt (\arr{a}{b})= (b)$. This shows that $\iso{H\XA^{1}}\cong \K \ct{G}_{D}$ as an algebra where $(p)$ are simply taking the place of $\id_{p}$ for $p\in V$. Proposition \ref{prop:quotHpfalg} and Lemma \ref{lemma:isoIdeal} tell us that the Hopf algebroid structure of $\iso{H\XA^{1}}$ is precisely the projection of the Hopf algebroid structure of $H\XA^{1}$ and it is easy to see from Corollaries \ref{CDiaBialg} and \ref{CDigAnti} that we recover the groupoid ring Hopf algebroid structure from Example \ref{ExGrpd}. 
\end{proof}
\begin{corollary}\label{CDXFund} The Hopf algebroid $\iso{\ct{D}\XA}$ is isomorphic to the groupoid algebra $\K \Pi_{D}$, where $\Pi_{D}$ is the fundamental groupoid of the digraph $D$ as described in the introduction.
\end{corollary}
\begin{proof} The proof is simply writing out the projection of the relations in Theorem \ref{TDXDigraph} in the isotopy quotient. In the last result we saw that generators of the form $(\genlefty{a}{b}{q})$ and $(\genrighty{a}{b}{q})$ vanish. Hence, relations \eqref{EqTri1} and \eqref{EqSqr1} simply vanish in $\iso{\ct{D}\XA}$. Observe that the expression $\mathbf{P}(p,q,r,a,b)$ simplifies in $\iso{H\XA^{1}}$ as 
\begin{align*}
\mathbf{P}(p,q,r,a,b)=& \sum_{i:\arrtwo{a}{i}{b}\in E_{2}}(\genleft{q}{r}{i}{b})\sbt (\genleft{p}{q}{a}{i})- \delta_{\arr{a}{b}}.\left( (\genlefty{q}{r}{b})\sbt (\genleft{p}{q}{a}{b})+ (\genleft{q}{r}{a}{b})\sbt (\genlefty{p}{q}{a})\right)
\\=&\sum_{i:\arrtwo{a}{i}{b}\in E_{2}}\delta_{q,i}.\delta_{r,b}(\genleft{q}{r}{i}{b})\sbt (\genleft{p}{q}{a}{i}).\delta_{a, p}= (\genleft{q}{r}{q}{r})\sbt (\genleft{p}{q}{p}{q})= (\arl{q}{r})\sbt (\arl{p}{q})
\end{align*}
Therefore, given a triangle $(p,q,r)$ in the digraph, relation \eqref{EqTri2} simplifies as follows 
$(\arl{q}{r})\sbt (\arl{p}{q})= (\arl{p}{r})$, while given a square $(p,q,q',r)$, relation \eqref{EqSqr2} simplifies to $(\arl{q}{r})\sbt (\arl{p}{q})= (\arl{q'}{r})\sbt (\arl{p}{q'})$. Hence, the relations on $\ct{D}\XA$ simplify to the precise relations defining the the fundamental groupoid of the digraph, when projected in the isotopy quotient, $\iso{\ct{D}\XA}$.
\end{proof}
\bibliographystyle{plain}
\bibliography{Hopf}
\end{document}